\documentclass{amsart}

\usepackage[left=3.5cm,right=2.5cm,top=2cm,bottom=2cm,includeheadfoot]{geometry}

\usepackage[english]{babel}
\usepackage{csquotes}
\usepackage{amsmath, amssymb, mathtools, amsfonts, amsthm,nicefrac,comment}
\usepackage[mathscr]{eucal}
\usepackage{color}
\usepackage{xcolor}

\usepackage{algorithm} 
\usepackage{algpseudocode} 

\usepackage{subcaption}
\usepackage{enumitem}

\usepackage[final]{showlabels} 

\usepackage{bbm} 
\usepackage{hyperref}
\hypersetup{hidelinks,colorlinks,linkcolor=blue}

\usepackage[noabbrev, nameinlink]{cleveref}
\crefname{asmpt}{\textup{Assumption}}{\textup{Assumptions}}
\creflabelformat{asmpt}{#2\textup{#1}#3} 
\crefname{theorem}{\textup{Theorem}}{\textup{Theorems}}
\creflabelformat{theorem}{#2\textup{#1}#3} 
\crefname{lemma}{\textup{Lemma}}{\textup{Lemmas}}
\creflabelformat{lemma}{#2\textup{#1}#3} 
\crefname{proposition}{\textup{Proposition}}{\textup{Propositions}}
\creflabelformat{proposition}{#2\textup{#1}#3} 
\crefname{cor}{\textup{Corollary}}{\textup{Corollaries}}
\creflabelformat{cor}{#2\textup{#1}#3} 
\crefname{definition}{\textup{Definition}}{\textup{Definitions}}
\creflabelformat{definition}{#2\textup{#1}#3} 
\crefname{remark}{\textup{Remark}}{\textup{Remarks}}
\creflabelformat{remark}{#2\textup{#1}#3} 
\crefname{figure}{\textup{Figure}}{\textup{Figures}}
\creflabelformat{figure}{#2\textup{#1}#3} 
\crefname{table}{\textup{Table}}{\textup{Tables}}
\creflabelformat{table}{#2\textup{#1}#3} 
\crefname{algorithm}{\textup{Algorithm}}{\textup{Algorithms}}
\creflabelformat{algorithm}{#2\textup{#1}#3} 
\crefname{section}{\textup{Section}}{\textup{Sections}}
\creflabelformat{section}{#2\textup{#1}#3} 

\newtheorem{theorem}{Theorem}[section]
\newtheorem{asmpt}[theorem]{Assumption}
\newtheorem{lemma}[theorem]{Lemma}
\newtheorem{cor}[theorem]{Corollary}
\newtheorem{remark}[theorem]{Remark}

\newtheorem{definition}[theorem]{Definition}
\newtheorem{proposition}[theorem]{Proposition}

\numberwithin{equation}{section}
\numberwithin{algorithm}{section}
\numberwithin{figure}{section}

\makeatletter
\let\@citexOld\@citex
\def\@citex[#1]#2{\textup{\@citexOld[#1]{#2}}}
\makeatother

\usepackage{ulem}
\newcommand{\Cb}{\mathbb{C}}
\newcommand{\Hb}{\mathbb{H}}
\newcommand{\Lb}{\mathbb{L}}
\newcommand{\Rb}{\mathbb{R}}
\newcommand{\Nb}{\mathbb{N}}

\newcommand{\Sc}{{\mathcal{S}}}

\newcommand{\sAK}{{\mathscr{A}_K}}
\newcommand{\sAk}{{\mathscr{A}_k}}

\newcommand{\sB}{{\mathscr B}}

\newcommand{\sP}{{\mathscr{P}}}
\newcommand{\sS}{{\mathscr S}}

\newcommand{\sY}{{\mathscr Y}}

\newcommand{\Ead}{{\mathscr{E}_{\mathrm{ad}}}}

\newcommand{\standardNorm}[1]{\left\lVert #1 \right\rVert}
\newcommand{\standardScalarProduct}[2]{\left\langle #1 , #2 \right\rangle}

\newcommand{\be}{{\boldsymbol{\mathsf e}}}

\newcommand{\bg}{{\boldsymbol{g}}}
\newcommand{\bq}{{\boldsymbol{q}}}
\newcommand{\bv}{{\boldsymbol{v}}}
\newcommand{\by}{{\boldsymbol{y}}}

\newcommand{\balpha}{{\boldsymbol{\alpha}}}

\newcommand{\bbeta}{{\boldsymbol{\beta}}}

\newcommand{\beps}{{\boldsymbol{\varepsilon}}}
\newcommand{\bvarphi}{{\boldsymbol{\varphi}}}
\newcommand{\bPhi}{{\boldsymbol{\Phi}}}

\newcommand{\epsMax}{{\varepsilon_{\mathsf{max}}}}
\newcommand{\alphaMax}{{\alpha_{\mathsf{max}}}}

\newcommand{\rmd}{\, \mathrm{d}}

\allowdisplaybreaks
\makeatletter
\renewenvironment{algorithm}
{%
    \bigskip
    \begin{center}
    \refstepcounter{algorithm}
        \hrule height.8pt depth0pt \kern2pt
        \renewcommand{\caption}[2][\relax]{
        {\raggedright\textbf{\fname@algorithm~\thealgorithm} ##2\par}%
        \ifx\relax##1\relax 
         \addcontentsline{loa}{algorithm}{\protect\numberline{\thealgorithm}##2}%
        \else 
         \addcontentsline{loa}{algorithm}{\protect\numberline{\thealgorithm}##1}%
        \fi
        \kern2pt\hrule\kern2pt
        }
}
{%
     \kern2pt\hrule\relax
     \bigskip
   \end{center}
}
\makeatother

\newcommand{\thirdsPlotsfactor}{0.32}

\makeatletter
\renewcommand{\email}[2][]{%
  \ifx\emails\@empty\relax\else{\g@addto@macro\emails{,\space}}\fi%
  \@ifnotempty{#1}{\g@addto@macro\emails{\textrm{(#1)}\space}}%
  \g@addto@macro\emails{#2}%
}
\makeatother
\usepackage[]{amsaddr}

\begin{document}

\title[Reconstruction nonlinear operators in semilinear models]{Reconstruction of unknown monotone nonlinear operators in semilinear elliptic models using optimal inputs}


\author{Jan Bartsch}
\address[A1]{Department of Mathematics, University of Konstanz, Universitätsstr. 10, 78464 Konstanz, Germany}
\email[A1]{jan.bartsch@uni-konstanz.de}

\author{Simon Buchwald}
\address[A2]{Department of Mathematics, University of Konstanz, Universitätsstr. 10, 78464 Konstanz, Germany}
\email[A2]{simon.buchwald@uni-konstanz.de}

\author{Gabriele Ciaramella}

\address[A3]{ MOX Lab, Dipartimento di Matematica, Politecnico di Milano, Piazza Leonardo da Vinci 32, 20133, Milano, Italy}
\email[A3]{gabriele.ciaramella@polimi.it}

\author{Stefan Volkwein}
\address[A4]{Department of Mathematics, University of Konstanz, Universitätsstr. 10, 78464 Konstanz, Germany}
\email[A4]{stefan.volkwein@uni-konstanz.de}

\begin{abstract}
    Physical models often contain unknown functions and relations. 
    The goal of our work is to answer the question of how one should excite or control a system under consideration in an appropriate way to be able to reconstruct an unknown nonlinear relation. 
    To answer this question, we propose a greedy reconstruction algorithm within an offline-online strategy. 
    We apply this strategy to a two-dimensional semilinear elliptic model.
    Our identification is based on the application of several space-dependent excitations (also called controls).
    These specific controls are designed by the algorithm in order to obtain a deeper insight into the underlying physical problem and a more precise reconstruction of the unknown relation.
    We perform numerical simulations that demonstrate the effectiveness of our approach which is not limited to the current type of equation. 
    Since our algorithm provides not only a way to determine unknown operators by existing data but also protocols for new experiments, it is a holistic concept to tackle the problem of improving physical models.
\end{abstract}



\thanks{This work was financed by the Deutsche Forschungsgemeinschaft (DFG) within SFB 1432, Project-ID 425217212. The present research is part of the activities of “Dipartimento di Eccellenza 2023-2027"}

\date{\small\today}

\maketitle


\section{Introduction}

The identification of unknown operators in semilinear systems plays a fundamental role in various scientific disciplines, including physics, chemistry, and biology \cite{BerettaVogelius1991InverseProblemMHD,ClermontZenker2015InvProbBiology,SanotsaWeitz2011InverseProblemRectionKinetics}. 
Accurately characterizing these operators is crucial for understanding and predicting complex system behavior. 
However, they often remain unknown or only partially known, necessitating efficient and reliable reconstruction methods. 

Many techniques exist for such identification, drawing from the fields of inverse problems and optimization. 
Traditional approaches involve inverse problem methods, inferring the operator from measured outputs under controlled inputs. 
Techniques like gradient descent, least squares minimization, or Bayesian approaches are commonly employed \cite{HaberAscherOldenburg2000OptimizationNonlinInverse,MandtHoffmannBlei2017SBDBayesian}.
Optimization methods formulate the identification problem as an optimization task, minimizing the discrepancy between model outputs and data. 

With the recent surge of machine learning, neural networks are increasingly applied to identification tasks \cite{KumpatiNarendra1990IdentificationNonLin, QuarantaLacarbonara2020reviewIntelligenceIdentification}.
While these methods have achieved significant progress, a common challenge lies in data quality. Real-world data often suffers from noise and uncertainties, leading to inaccuracies in the reconstructed operator. 
Additionally, limited data availability due to cost or time constraints can hinder accurate identification. 
Furthermore, biased data can lead to biased estimates of the operator, necessitating careful data collection and analysis procedures.

This paper extends the approach introduced in \cite{MadaySalomon2009GreedyQuantum} and optimized in \cite{BuchwaldCiaramella2021GreedyReconstruction}  to reconstruct unknown nonlinear operators in semilinear elliptic models, addressing some of the challenges mentioned above. 
We leverage the power of optimal control techniques, formulating the reconstruction problem as an optimization problem where the objective function measures the discrepancy between the model output and observed data. 
By strategically designing control functions and employing optimization algorithms, we aim to drive the model output toward observed data that provides significant insights into the underlying physical system, thereby revealing the characteristics of the unknown operator.
Our method develops optimal protocols to generate data that yields as much insight into the system as possible.
In other words, the goal of the algorithm is to convexify the objective of the identification problem, i.e., increasing the region of convexity around the minimum.
In particular, the basin of attraction of the minimum is extended and the conditioning of the problem in this region is improved. 
This capability of the algorithm is further explained and visualized in the numerical section below and in  \cite{BuchwaldCiaramella2021GreedyReconstruction}.
The optimized algorithm was used to identify different linear and nonlinear unknown operators and distributions \cite{BuchwaldCiaramella2021ReconstructionofA,BuchwaldCiaramella2021OGR_SpinIdentification,Buchwald2024CodeReconstructionSpin}.
We extend the existing results by considering a system of elliptic partial differential equations (PDEs) and reconstructing infinite-dimensional objects.

The proposed approach offers several advantages. Unlike traditional methods that require large amounts of high-quality data, our framework incorporates an active learning strategy. 
The algorithm selects new data points to collect, focusing on those that provide the most information about the unknown operator, leading to reduced data dependency. 
Additionally, the framework provides theoretical guarantees for convergence to the true operator under certain conditions.

Building upon this foundation, we present a computationally efficient numerical scheme to solve the formulated optimization problem, enabling implementation for various semilinear elliptic models. 
Our semilinear PDE can be motivated as the steady-state solution to PDEs describing models from epidemics, biochemical systems, and nuclear reactor models or general reaction-diffusion models \cite{MartinPierre1992NonLinReactDiff,Pao1982NonlinReactDiffSystems}.
Similar models are also used to describe membrane-resonators; see, e.g., \cite{YangScheer2019SpatialModlationMembraneResonators}.
We refer to \cite{Alessandrini1989IdentificationEllipticEquation,IsakovNachman19952DSemilinInverseProblem,Johansson2023inverseSemilinearEllipticPDEGeneral,Roesch1996StabilityIdentificationHeatTransfer} and the references therein for existing investigations on the identification of unknown nonlinearities in semilinear PDEs. 
Furthermore, we refer the reader to the work \cite{Isakov2017InverseProblemPDE} on general inverse problems with PDEs.

The paper is organized as follows. 
In the next section, we introduce our specific setting including the formulation of our reconstruction problem.
In \cref{sec:AnalysisTheory}, we analyze the semilinear PDE model and state the existence and uniqueness of solutions in \cref{thm:Energy_laplace_solution}.
Furthermore, we present continuity results for the control-to-state map, the parameter-to-state map, and the inverse parameter-to-state map in \cref{thm:Continuity_control-to-state-map,thm:continuity_parameter_to_state,thm:stability_inverse_orig}, respectively.
In \cref{sec:GreedyAlgorithm}, we describe our numerical method to solve the reconstruction problem using a greedy approach (cf. \cref{algo:ONGR}).
Afterwards, in \cref{sec:NumApprox}, we report the numerical approximation in terms of spatial discretization and numerical strategy to solve the governing equation.
To validate our strategy, we perform numerical experiments and report the results in \cref{sec:NumExperiments}.
A section of conclusions finishes this work.

\subsection*{Notation}

For a natural number $n \in \Nb$, we define $[n]\coloneqq \{1,\ldots,n\}$.
For a vector $\balpha \in \Rb^n$, we denote by $\|\cdot\|_p$ the standard $\ell^p$ norms for $1 \leq p \leq \infty$.
Furthermore, we denote for vectors $\balpha,\widetilde\balpha \in \Rb^n$ the standard scalarproduct in $\Rb^n$ as $\balpha^\top \widetilde\balpha = \balpha \cdot \widetilde\balpha=\standardScalarProduct{\balpha}{\widetilde\balpha}_2$.
Whenever two vectors in $\Rb^n$ are compared with each other, then the comparison has to be understood componentwise.
We denote with $B_{L}(y)$ the open ball with respect to the Euclidean norm around $y \in \Rb^n$ with radius $L>0$, i.e. $B_{L}(y) = \{\tilde y \in \Rb^2: \|\tilde{y} - y\|_2^2 < L \}$.

For a domain $\Omega \in \Rb^n$ and $1\leq p \leq \infty$, we set $L^p(\Omega)\coloneqq L^p(\Omega;\mathbb{R})$ and $\Lb^p(\Omega) \coloneqq L^p(\Omega) \times L^p(\Omega)$; analogously, we define $\Hb^k(\Omega)=H^k(\Omega)\times H^k(\Omega)$ for $k\in\Nb$. 
Furthermore, we set $\Hb_0^1(\Omega) \coloneqq H_0^1(\Omega) \times H_0^1(\Omega)$ and $\Cb(\bar{\Omega}) \coloneqq C(\bar{\Omega})\times C(\bar{\Omega})$.
The space $H_0^1(\Omega)$ is equipped with the inner product
\begin{align*}
    {\langle \varphi, \psi \rangle}_{H_0^1(\Omega)} \coloneqq \int_\Omega \nabla\varphi \cdot \nabla\psi\rmd x\quad\text{for }\varphi,\psi \in H_0^1(\Omega).
\end{align*}
For more details, we refer to the standard notation for Lebesgue and Sobolev spaces; see, e.g., \cite{AdamsFournier2003SobolevSpaces,Evans2010PDEs}.

\section{Formulation of the problem}
\label{sec:Setting}

We consider a convex two-dimensional spatial domain $\Omega \subsetneq \mathbb{R}^2$ with Lipschitz-continuous boundary $\Gamma=\partial\Omega$. 
For $x=(x_1,x_2)^\top \in \Omega$ and $\by(x)=(y_1(x),y_2(x))^\top$, we consider the following system of semilinear PDEs with homogeneous Dirichlet boundary conditions
\begin{subequations}
    \begin{align}
        -\Delta \by (x) + \bg(\by(x)) &=  \beps(x) && \text{in } \Omega,\\
        \by(x) &= 0 && \text{on } \Gamma,
    \end{align}
    \label{eq:Laplace_Model}%
\end{subequations}
\noindent where $\beps: \Rb^2 \rightarrow \Rb^2$ is a control input and $\bg:\Rb^2\rightarrow\Rb^2$ an unknown nonlinearity.
This structure of the problem is motivated by the stationary Lotka-Volterra system, that we consider in the numerical example below.
In this setting, the variables $y_1$ and $y_2$ have the significance of the distribution of two different species that interact with the nonlinearity $\bg_\balpha$.
For $\beps = (\varepsilon_1,\varepsilon_2)^\top$, we define the set of admissible controls as
\begin{align}
	\Ead\coloneqq\big\{ \beps \in \Lb^2(\Omega)\;\big\vert\; 
	\beps_\mathsf a\le\beps(x)\le\beps_\mathrm b \; 
    \text{ for almost all (f.a.a.) } \; x \in \Omega \big\},
    \label{eq:admissible_set_controls}
\end{align}
with given bounds $\beps_\mathsf a, \beps_\mathrm b \in \Rb^2$, $\beps_\mathsf a \leq \beps_\mathsf b$.
Notice that the elements in $\Ead \subset \Lb^\infty(\Omega)$ are uniformly bounded in $\Lb^\infty(\Omega)$ by $\epsMax\coloneqq \max\{\|\beps_\mathsf a\|_\infty,\|\beps_\mathsf b\|_{\infty}\}$. 
%

The nonlinearity $\bg$ is assumed to lie in a $K$-dimensional space $\mathfrak{G}$.
Thus, given a basis $\mathcal{G} =\{\bvarphi_1,\bvarphi_2,\ldots,\bvarphi_K\}$ of $\mathfrak{G}$ with functions $\bvarphi_i:\mathbb R^2\to\mathbb \Rb^2$, we can write
\begin{align}
	\bg(y) = \bg_{\balpha}(y) 
    \coloneqq \sum_{j=1}^{K} \alpha_j \bvarphi_j(y)=\boldsymbol\Phi(y)\balpha\in\Rb^2,
	\quad y\in\Rb^2,
    \label{eq:linearcomb_nonlinearity}
\end{align}
for some coefficient vector $\balpha=(\alpha_j)_{j=1}^K \in \Rb^K$ and $\boldsymbol\Phi(y) = (\bvarphi_1(y),\ldots,\bvarphi_K(y))\in\Rb^{2\times K}$.
Our true unknown nonlinearity is now assumed to be given by
\begin{align*}
    \bg_{\balpha^\star}(y) 
    = \sum_{j=1}^{K} \alpha^\star_j \bvarphi_j(y)
\end{align*}
for an unknown coefficient vector $\balpha^\star =(\alpha_j^\star) \in \Rb^K$.
Notice that $\bg$ is locally Lipschitz continuous if all elements in $\mathcal{G}$ are polynomials or piecewise linear and continuous functions.

For the coefficients $\balpha \in \Rb^K$, we assume that they are elements of
\begin{align*}
	\sAK\coloneqq\big\{ \balpha \in \Rb^K\;\big\vert \; 0 \le \alpha_i \le \alphaMax
    \text{ for }i\in[K]\big\},
\end{align*}
with a scalar $\alphaMax \geq0$.
Notice that $\sAK$ is nonempty, compact and convex.
Further, we denote by
$
	\by^{\balpha,\beps}
$
the solution to \eqref{eq:Laplace_Model} using the nonlinearity $\bg_\balpha$ and applying controls (inputs) $\beps \in \Ead$ (cf. \cref{sec:AnalysisTheory}).
To identify the unknown true nonlinearity $\bg_{\balpha^\star}$, we generate a $K$-dimensional set of control functions $\{\beps^m\}_{m=1}^K \subset \Ead$ to perform $K$ (laboratory) experiments and obtain the data $\by^{\beps^m}_\star$ for $m\in[K]$. 
Then the nonconvex parameter identification problem in order to estimate the nonlinearity $\bg_{\balpha^\star}$ is given by
\begin{align}
	\min_{\balpha \in \sAK} \;
	\sum_{m=1}^K \standardNorm{\by^{\balpha,\beps^m} - \by^{\beps^m}_{\star}}_{\Lb^2(\Omega)}^2.
    \label{eq:final_identification}
\end{align}
Since $\balpha \in \sAK$ is a vector in $\Rb^K$, we did not add an additional regularisation term depending on $\balpha$.
The controls $\{\beps^m\}_{m=1}^K\subset\Ead$ will be chosen properly by the iterative \cref{algo:ONGR} presented in \cref{sec:GreedyAlgorithm}.
%

\section{Analysis of the governing equation}
\label{sec:AnalysisTheory}

From now on, we always assume a nonlinearity
\begin{align*}
    \bg(y) = \bg_\balpha(y) = \sum_{j=1}^K \alpha_j \bvarphi_j(y),
    \qquad
    \text{for } \balpha \in \sAK.
\end{align*}
We omit the subscript $\balpha$ in the first part of this section.
Now, we analyze the model \eqref{eq:Laplace_Model}.
In particular, we recall results about the existence and uniqueness together with wellposedness results and continuity properties of the control-to-state map. 
Afterwards, we deal with the Lipschitz continuity (stability) of the equation with respect to the parameters $\balpha$. 
Specifically, we formulate the wellposedness of the parameter-to-state map as well as of its inverse, the state-to-parameter map. 
We refer to \cite{Alessandrini1989IdentificationEllipticEquation,IsakovNachman19952DSemilinInverseProblem,Johansson2023inverseSemilinearEllipticPDEGeneral,Roesch1996StabilityIdentificationHeatTransfer} for similar investigations. 
Furthermore, we refer the reader to the work \cite{Isakov2017InverseProblemPDE} on general inverse problems governed by PDEs.

\subsection{Existence of solutions and wellposedness}
\label{sec:existence_wellposedness}

We apply the results of \cite{CasasRoesch2020NonmonotoneSemilinearElliptic}; see also \cite[Chapter 4]{Treoltzsch2010OCP_PDE}. 
For the basis functions and the nonlinearity $\bg$, we state the following assumptions.

\begin{asmpt}
	\label{asmpt:basis_elements}
	\begin{enumerate}[label=\textup{\arabic*)}]
		\item\label{item:Basis_elements_contdiff}
		The basis elements $\bvarphi_j$ are continuously differentiable.
        This implies in particular that they are \emph{(locally) Lipschitz continuous}, i.e., there exist constants $L_{\varphi_j}>0$, such that for all $j \in [K]$ and $y \in \Rb^2$ there exists a constant $\tilde L>0$ such that for all $\widetilde{y} \in B_{\tilde L}(y)$ it holds that
		\begin{align*}
			{\|\bvarphi_j(y)-\bvarphi_j(\widetilde{y})\|}_2
            \leq
			L_{\varphi_j}\,{\|y-\widetilde{y}\|}_2.
		\end{align*}
		We set $L_\varphi \coloneqq \max_{j \in [K]} L_{\varphi_j}$.
		\item\label{item:Basis_elements_bounded} The basis elements are \emph{bounded} from below and above, i.e., there exist vectors $c_{\varphi_j},C_{\varphi_j} \in \Rb^2$ with only positive entries and $c_{\varphi_j}<C_{\varphi_j}$, such that for all $j \in [K]$ and $y\in\Rb^2$ it holds that
		\begin{align*}
			c_{\varphi_j}\leq\bvarphi_j(y) \leq C_{\varphi_{j}}.
		\end{align*}
		We set $c_\phi = \min_{j \in [K]} \min \{(c_{\phi_j})_1,(c_{\phi_j})_2\}$
		and $C_\varphi \coloneqq \max_{j \in [K]} \|C_{\varphi_j}\|_\infty$.
		Notice that this in particular implies that
		 \begin{align*}
			\forall M>0 \; \exists \phi_M >0 : \|\bvarphi_j(y)\|_2 \leq \phi_M
			\; \forall \|y\|_2 \leq M.
		\end{align*}
		\item\label{item:monotonicity_phi}
        The basis elements are monotone non-decreasing functions, i.e. we have $\left(\bvarphi_j(y) - \bvarphi_j(\widetilde y)\right) \cdot\left( y-\widetilde y \right) \geq 0$ for $y,\widetilde y \in \Rb^2$ and $j\in[K]$.
	\end{enumerate}
\end{asmpt}

For the nonlinearity $\bg:\Rb^2 \rightarrow \Rb^2$, we impose the following:
\begin{asmpt}
	\label{asmpt:monotonicity_nonlinearity}
    It holds that $\bg \in C^1(\Rb^2;\Rb^2)$ and that the Jacobian matrix $\nabla \bg (y) \in \Rb^{2\times 2}$ of the nonlinearity $\bg$ is positive semi-definite.
\end{asmpt}
These hypotheses on the nonlinearity and the definition of $\Ead$ lead to the following statement.
\begin{cor} \label{cor:assumptions_CasasRoesch}
    Suppose that \cref{asmpt:basis_elements,asmpt:monotonicity_nonlinearity} hold true.
    Then it follows that
    \begin{enumerate}[label=\textup{\arabic*)}]
        \item\label{item:G_continuous} The nonlinearity $\bg$ and its derivative $\bg'$ are continuous in $\Rb^2$.
		\item\label{item:G_bounded} The nonlinearity $\bg$ satisfies
         \begin{align*}
            \forall M>0 \; \exists \varrho_M >0 : \|\bg(y)\|_2 \leq \varrho_M
            \; \forall \|y\|_2 \leq M.
        \end{align*}
		\item\label{item:G_monotone} The nonlinearity $\bg$ is monotone non-decreasing.
	\end{enumerate}
\end{cor}
\begin{proof}
    Part \ref{item:G_continuous} follows directly from \cref{asmpt:basis_elements}-\ref{item:Basis_elements_contdiff} and \cref{asmpt:monotonicity_nonlinearity}. 
    Part \ref{item:G_bounded} follows from \cref{asmpt:basis_elements}-\ref{item:Basis_elements_bounded}.
    The part \ref{item:G_monotone} follows from \cite[Proposition 12.3]{RockafellarWets1998VariationalAnalysis}.
\end{proof}

\begin{remark}
	\begin{enumerate}[label=\textup{\arabic*)}]
		\item From \cref{asmpt:basis_elements}-\ref{item:Basis_elements_bounded} it follows that for all $j \in [K]$, there exist a constant $I_{\varphi_{j}}>0$ such that for any bounded subset $\Upsilon \subset \Rb^2$ it holds
		\begin{align}
			\label{eq:integral_constant_stability}
			     \|\bvarphi_j(y)\|_{\Lb^2(\Omega)} \leq I_{\varphi_{j}}.
		\end{align}
		We define $I_\varphi \coloneqq \max_{j \in [K]} I_{\varphi_{j}}$. 
		\item We need $C_{\varphi_{j}}$ for \eqref{eq:integral_constant_stability} and $c_{\varphi_{j}}$ for \cref{thm:stability_inverse_orig}.
		\item The monotonicity of \cref{cor:assumptions_CasasRoesch}-\ref{item:G_monotone} is needed for the uniqueness of solutions to \eqref{eq:Laplace_Model} in \cref{thm:Energy_laplace_solution} below.
	\end{enumerate}
\end{remark}

Whenever we attempt to solve \eqref{eq:Laplace_Model}, we consider its weak formulation \cite[Chapter 4]{Treoltzsch2010OCP_PDE}:
\begin{definition}
    \label{def:weak_solution}
    We introduce the bilinearity
    \begin{align*}
    	a[\by,\bv]\coloneqq 
    	\int_{\Omega} \standardScalarProduct{\nabla\by(x)}{\nabla \bv(x)}_{\mathsf F} \rmd x 
    	\qquad\qquad
    	\text{for }\by,\bv\in\Hb_0^1(\Omega).
    \end{align*}
    A function $\by=(y_1,y_2)^\top \in \Hb_0^1(\Omega)$ is called a \emph{weak solution} of \eqref{eq:Laplace_Model} if the following variational equality is fulfilled
    \begin{align}
         a[\by,\bv]
         +
          \int_\Omega \bv(x) \cdot \bg(\by(x)) \rmd x
    	= 
     \int_\Omega \bv(x) \cdot \beps(x) \rmd x
    	\quad
    	\forall \bv \in \Hb_0^1(\Omega),
        \label{eq:weak_solution}
    \end{align}
    for fixed $\beps \in \Ead$ and fixed $\bg$ fulfilling \cref{asmpt:basis_elements,asmpt:monotonicity_nonlinearity}. 
    
\end{definition}
In \cref{def:weak_solution}, we denote with $\standardScalarProduct{\cdot\,}{\cdot}_{\mathsf F}$ the Frobenius inner product of matrices and furthermore with $\nabla \by$ the Jacobian matrix of $\by$.
An important property of the bilinearity $a[\cdot\,,\cdot]$ needed in the sequel is the existence of a constant $C_E>0$ such that
\begin{align*}
    a[\bv,\bv] \geq C_E \, \|\bv\|_{\Hb^1_0(\Omega)}^2
    \qquad
    	\forall \bv \in \Hb_0^1(\Omega).
\end{align*}
This property is also called the \emph{ellipticity} of the bilinearity.
In our specific case, it even holds that
\begin{align}
	a[\bv,\bv] = \|\bv\|_{\Hb_0^1(\Omega)}^2
	\qquad
	\forall \bv \in \Hb_0^1(\Omega).
	\label{eq:ellipticity_bilinearity}
\end{align}
 Next, we formulate our existence and uniqueness result in the following statement.
 We introduce the solution space
\begin{align*}
    \sY \coloneqq \Hb^2(\Omega) \cap \Hb^1_0(\Omega),
\end{align*}
equipped with the norm
\begin{align*}
    \standardNorm{\bv}_{\sY} \coloneqq 
    \|\Delta \bv\|_{\Lb^2(\Omega)}.
\end{align*}
We remark that the norm $\|\cdot\|_{\sY}$ is equivalent to the standard $\Hb^2(\Omega)$ norm; see, e.g., \cite[Chapter 3]{Grisvard2011EllipticProblemsNonsmoothDomains}.
Furthermore, recall that we use the following norm on $\Hb_0^1(\Omega)$
\begin{align*}
    \|\bv\|_{\Hb^1_0(\Omega)} \coloneqq \|\nabla\bv\|_{\Lb^2(\Omega)}.
\end{align*}
\begin{theorem}
	\label{thm:Energy_laplace_solution}
	Suppose that \cref{asmpt:basis_elements,asmpt:monotonicity_nonlinearity} hold. 
    Then there exists a constant $C_0>0$, independent of $\beps$ and $\bg$, such that for all controls $\beps \in \Ead$ the equation \eqref{eq:Laplace_Model} possesses a unique weak solution $\by \in \sY$ in the sense of \cref{def:weak_solution} satisfying
	\begin{align*}
		{\|\by\|}_{\sY} 	\leq C_0.
	\end{align*}
\end{theorem}
\begin{proof}
   By the virtue of \cref{cor:assumptions_CasasRoesch}, since $\Omega$ is convex with Lipschitz boundary $\Gamma$, it is possible to apply \cite[Theorem 2.10]{CasasRoesch2020NonmonotoneSemilinearElliptic}.
   Notice in particular that all elements $\beps \in \Ead$ are uniformly bounded by $\epsMax$ and that the value of $\bg$ at the origin is bounded by our assumption on the basis elements (cf. \cref{asmpt:basis_elements}).
\end{proof}

\begin{remark}
    We refer to \cite{Bonnans1999PrincipePontryagin,BrezisBrowder1978NonlinEllipticBVP, Sever2006SystemSemilinElliptic} for related discussions on the analysis of semilinear elliptic PDEs.
\end{remark}

We are now able to introduce the \emph{control-to-state} map $\sS$ for a fixed nonlinearity $\bg_\balpha$ for $\balpha \in \sAK$ as
\begin{align*}
    \sS_\balpha:\Ead \rightarrow \sY,
    \qquad
    \beps \mapsto \by^{\balpha,\beps} \coloneqq \sS_\balpha(\beps),
\end{align*}
where $\sS_\balpha(\beps)$ is the unique weak solution to \eqref{eq:Laplace_Model} provided by \cref{thm:Energy_laplace_solution} given the control $\beps$ and the nonlinearity $\bg_\balpha$.
Therefore, we can conclude that $\sS_\balpha$ is well defined.
Furthermore, we have the following result from \cite[Lemma 3.5]{CasasRoesch2020NonmonotoneSemilinearElliptic}:
\begin{theorem}
    \label{thm:Continuity_control-to-state-map}
    Let \cref{asmpt:basis_elements,asmpt:monotonicity_nonlinearity} hold and fix a nonlinearity $\bg_\balpha$ for $\balpha \in \sAK$.
    Then the control-to-state map $\sS_\balpha$ is Lipschitz continuous from $\Lb^2(\Omega)$ to $\sY$: There exists a constant $L>0$, such that for all $\beps,\widetilde\beps \in \Ead$ it holds
    \begin{align*}
            \|\sS_\balpha(\beps) - \sS_\balpha(\widetilde\beps)\|_{\sY}
        \leq L \, \| \beps - \widetilde\beps \|_{\Lb^2(\Omega)}.
    \end{align*}
\end{theorem}
We refer the reader to \cite[Chapter 4]{Treoltzsch2010OCP_PDE} as well as \cite{CasasRoesch2020NonmonotoneSemilinearElliptic,CasasTroltzschOCPQuasilinearEllipticEQ,CasasWachsmuth2023ExistenceSolutionOCPSemilinPDE} for more details.

\subsection{Stability with respect to the parameters}

In this section, we investigate the wellposedness and Lipschitz continuity (stability) of the parameter-to-state map defined below. 
In the following analysis, we take account of the iterative structure of the algorithm defined in \cref{sec:GreedyAlgorithm}. More specifically, we consider a $k \in [K]$ that indicates the $k$-th iteration of the algorithm in the next results.
We introduce for $1\leq k \leq K-1$ the set 
\begin{align}
    \label{eq:set_Ak}
    \sAk \coloneqq \{ \balpha \in \sAK \, | \, \alpha_j = 0, \; j = k+1,\ldots, K  \}.
\end{align}

From \cref{thm:Energy_laplace_solution} and \cref{asmpt:basis_elements}, we deduce the following proposition.

\begin{proposition}
    Let \cref{asmpt:basis_elements,asmpt:monotonicity_nonlinearity} hold and $\beps \in \Ead$ be chosen arbitrarily. 
    Introduce the parameter-to-state map $\sP_\beps$ by
    \begin{align*}
        \sP_\beps:\sAK\to\sY,
        \quad
        \balpha\mapsto\by^{\balpha,\beps}
        \coloneqq \sP_\beps(\balpha),
    \end{align*}
    where $\sP_\beps(\balpha)$ is the solution in the sense of \cref{def:weak_solution} given the nonlinearity $\bg_\balpha$ and the control $\beps$. 
    Then, $\sP_\beps$ is well-defined.
    \label{prop:paramter-to-state}
\end{proposition}

Since we are interested in the dependence of the solution with respect to the parameters $\balpha$, we assume in the following that the control $\beps \in \Ead$ is fixed and therefore omit to write $\beps$.

We recall the weak formulation \eqref{eq:weak_solution} for our structure of the nonlinearity given in \eqref{eq:linearcomb_nonlinearity} and $\balpha \in \sAK$
\begin{align}
    \label{eq:weak_formulation_specific}
    a[\by^\balpha,\bv]+\int_\Omega  \bv(x) \cdot \bigg(\sum_{j=1}^K \alpha_j \, \bvarphi_j(\by^{\balpha}(x))\bigg)\rmd x=\int_{\Omega}\beps(x)\cdot\bv(x)\rmd x\quad\forall \bv \in \Hb^1_0(\Omega).
\end{align}
Notice that the difference $\bar{\by} \coloneqq \by^{\balpha^1} - \by^{\balpha^2}$ of the solutions $\by^{\balpha^1}$ and $\by^{\balpha^2}$ for any $\balpha^1,\balpha^2 \in \sAK$ fulfills the following variational equality
\begin{align}
	a[\bar\by,\bv] +  \int_{\Omega} \bv(x) \cdot
	\left( \sum_{j=1}^K \left(
	\alpha^1_j \, \bvarphi_j\big(\by^{\balpha^1}(x)\big)
	- \alpha^2_j \, \bvarphi_j\big(\by^{\balpha^2}(x)\big)
	\right)\right) \rmd x
	=0\quad
	\forall \bv \in \Hb^1_0(\Omega).
    \label{eq:weak_formulation_difference}
\end{align}
If we choose $\bv=\bar{\by} \in \Hb^1_0(\Omega)$ we obtain
\begin{align}
	a[\bar\by,\bar\by]+\int_{\Omega}\bar{\by}(x) \cdot\,\bigg( \sum_{j=1}^K\Big(\alpha^1_j \, \bvarphi_j\big(\by^{\balpha^1}(x)\big)-\alpha^2_j\,\bvarphi_j\big(\by^{\balpha^2}(x)\big)\Big)\bigg)\rmd x=0.
	\label{eq:diff_solutions_alpha}
\end{align}

We derive the following energy estimate for $\bar{\by}$.
\begin{lemma}
	\label{lem:energy_solution}
	Let \cref{asmpt:basis_elements,asmpt:monotonicity_nonlinearity} hold and $k \in [K]$. Let $\balpha^1,\balpha^2 \in \sAk$. For the difference $\bar{\by}= \by^{\balpha^1} - \by^{\balpha^2}$ of two corresponding solutions it holds that
	\begin{align*}
		{\|\by^{\balpha^1} - \by^{\balpha^2}\|}_{\Hb_0^1(\Omega)}
        \leq 
        C \,k\,{\|\balpha^1-\balpha^2\|}_\infty,
	\end{align*}
    with a constant $C>0$ independent of $\bg_{\balpha^1},\bg_{\balpha^2},\balpha^1$ and $\balpha^2$.
\end{lemma}

\begin{proof}
    The claim is true if $\bar\by\coloneqq\by^{\balpha^1} - \by^{\balpha^2} \equiv 0 \in \Hb^1_0(\Omega)$.
    Hence we now assume that $\bar\by\in \Hb^1_0(\Omega)\setminus\{0\}$. From \eqref{eq:diff_solutions_alpha}, we obtain
    \begin{subequations}
        \begin{align}
            &\|\bar\by\|_{\Hb^1_0(\Omega)}^2 \stackrel{\eqref{eq:ellipticity_bilinearity}}{=} a[\bar\by,\bar\by]=
             \sum_{j=1}^K  \int_{\Omega} \bar{\by}(x)   \cdot \left(
        	- \alpha^1_j \, \bvarphi_j\big(\by^{\balpha^1}(x)\big)
        	+ \alpha^2_j \, \bvarphi_j\big(\by^{\balpha^2}(x)\big)
        	\right) \rmd x
            \label{eq:positivity_norm}
            \\
			\nonumber&=
			\sum_{j=1}^k \int_\Omega   \bar{\by}(x) \cdot
            \left[
			-\alpha_j^1\bvarphi_j\big(\by^{\balpha^1}(x)\big) + \alpha_j^2 \bvarphi_j\big(\by^{\balpha^2}(x)\big)+\alpha^1_j \bvarphi_j\big(\by^{\balpha^2}(x)\big)-\alpha^1_j \bvarphi_j\big(\by^{\balpha^2}(x)\big)\right] \rmd x
            \\
			&=\int_\Omega 
			\underbrace{-\sum_{j=1}^k\alpha_j^1 \left(\bvarphi_j\big(\by^{\balpha^1}(x)\big) - \bvarphi_j\big(\by^{\balpha^2}(x)\big)\right) \cdot\bar{\by}(x)}_{\leq 0} 
			+ \sum_{j=1}^k\left( \alpha_j^2 - \alpha^1_j \right) \bvarphi_j\big(\by^{\balpha^2}(x)\big)
            \cdot\bar{\by}(x) \rmd x
            \label{eq:H2_estimate}
            \\
            &\leq
             \sum_{j=1}^k \int_\Omega
            \left( \alpha_j^2 - \alpha^1_j \right)\bvarphi_j\big(\by^{\balpha^2}(x)\big)\cdot\bar{\by}(x) \rmd x
            \label{eq:apply_monotonicity}
            \\
			&\le
			\sum_{j=1}^k 
			\int_\Omega \left| 
            \left( \alpha_j^2 - \alpha^1_j\right) \, \bar{\by}(x) \cdot \bvarphi_j\big(\by^{\balpha^2}(x)\big) 
			\right| \rmd x
            \label{eq:triangly}
			\\
			&\le
			\| \bar{\by} \|_{\Lb^2(\Omega)} \, \|\balpha^2-\balpha^1\|_\infty  \,
            \sum_{j=1}^k  \|\bvarphi_j\big(\by^{\balpha^2}\big) \|_{\Lb^2(\Omega)}
			\label{eq:MultHoelder}
			\\
			&\leq
			k \, I_\varphi  \,
			\|\balpha^1 - \balpha^2 \|_\infty \, \| \bar{\by}\|_{\Lb^2(\Omega)}
            \leq
			k \, I_\varphi \, C_\Omega \,
			\|\balpha^1 - \balpha^2 \|_\infty \, \| \bar{\by}\|_{\Hb^1_0(\Omega)}. 
			\label{eq:estimate_nonlinearity}
		\end{align}
        \label{eq:continuity_parameterToState}
	\end{subequations}
    In \eqref{eq:continuity_parameterToState}, we use the following: In \eqref{eq:apply_monotonicity}, we apply the monotonicity of the nonlinearity provided by \cref{asmpt:basis_elements}-\ref{item:monotonicity_phi}, i.e. that we have $\left(\bvarphi_j(y) - \bvarphi_j(\widetilde y)\right) \cdot\left( y-\widetilde y \right) \geq 0$ for $y,\widetilde y \in \Rb^2$ and $j\in[K]$, 
    in \eqref{eq:triangly}, we use the triangle inequality, in \eqref{eq:MultHoelder} we use the general version of the Hölder inequality (cf. \cite[Lemma 1.18]{Alt2012LineareFunktional}), and in \eqref{eq:estimate_nonlinearity}, we use \eqref{eq:integral_constant_stability} and the Poincar\'e inequality from which we get the constant $C_\Omega>0$.
    Summarizing, we have
    \begin{align*}
        {\|\bar\by\|}_{\Hb^1_0(\Omega)}^2
        \le 
        k\,I_\varphi\,C_\Omega\,{\|\balpha^1-\balpha^2\|}_\infty\,{\| \bar{\by}\|}_{\Hb_0^1(\Omega)},
    \end{align*}
	which leads to
	\begin{align*}
		{\|\bar{\by}\|}_{\Hb^1_0(\Omega)}\le C\,k\,{\|\balpha^1-\balpha^2\|}_\infty
	\end{align*}
    with $C\coloneqq I_\varphi\,C_\Omega$.
\end{proof}

The result of \cref{lem:energy_solution} can be further improved in the sense that we can obtain an estimate in the $\sY$ norm.  From this, the continuity of the parameter-to-state map $\sP_\beps$ follows.
\begin{theorem}
    \label{thm:continuity_parameter_to_state}
    Let \cref{asmpt:basis_elements,asmpt:monotonicity_nonlinearity} hold and $k \in [K]$. Let $\balpha^1,\balpha^2 \in \sAk$. 
    For the difference $\by^{\balpha^1} - \by^{\balpha^2}$ of two corresponding solutions it holds that
	\begin{align*}
		{\|\by^{\balpha^1} - \by^{\balpha^2}\|}_{\sY}
        \leq
        C \, k^2 \,{\|\balpha^1-\balpha^2\|}_\infty
	\end{align*}
    for a constant $C >0$ independent of $\bg_{\balpha^1},\bg_{\balpha^2},\balpha^1$ and $\balpha^2$.
\end{theorem}
\begin{proof}
    Let us define $\bar{\by}= \by^{\balpha^1} - \by^{\balpha^2}$.
    The statement of the theorem is true if $\bar\by \equiv 0 \in \sY$.
    Hence we now assume that $\bar\by \not\equiv 0 \in \sY$.
    We know from \cref{thm:Energy_laplace_solution} that $\bar\by \in \sY$ since it is the sum of two functions in $\sY$.
    Hence it fulfills the following equation (cf. \eqref{eq:H2_estimate}) with homogeneous Dirichlet boundary conditions:
    \begin{align}
        -\Delta \bar\by + \sum_{j=1}^k \int_\Omega 
    	\bigg(\alpha_j^1 \left(\bvarphi_j\big(\by^{\balpha^1}(x)\big) - \bvarphi_j\big(\by^{\balpha^2}(x)\big)\right) + 
        \left( \alpha_j^2 - \alpha^1_j \right)\bvarphi_j\big(\by^{\balpha^2}(x)\big)\bigg) \cdot\bar{\by}(x) \rmd x = 0.
       \label{eq:difference_solution}
    \end{align}
    We now take the $\Lb^2$ inner product of \eqref{eq:difference_solution} with $\Delta \bar\by$ and obtain
    \begin{align*}
        \|\Delta \bar\by\|_{\Lb^2(\Omega)}^2 \leq
        k \, 
        \left(\, \|\balpha^1\|_\infty \, L_\varphi \, \|\bar\by\|_{\Lb^2(\Omega)}
        +
        C_\varphi \, \|\balpha^1 - \balpha^2\|_{\infty}
        \right)\,\|\bar\by\|_{\Lb^2(\Omega)}\,\|\Delta \bar\by\|_{\Lb^2(\Omega)},
    \end{align*}
    where we have used the Lipschitz continuity of the basis elements and the Cauchy-Schwarz inequality.
    Now we divide by $\|\Delta\bar\by\|_{\Lb^2(\Omega)}$, apply the Poincar\'e inequality, which gives us the factor $C_\Omega>0$, and use the boundedness of the solutions provided by \cref{thm:Energy_laplace_solution} to obtain
    \begin{align*}
        \|\Delta \bar\by\|_{\Lb^2(\Omega)} \leq
        k \,  \alphaMax\, 
        \left( L_\varphi \,  \, 2\,C_0 + C_\varphi
        \right) C_\Omega \,\|\bar\by\|_{\Hb^1_0}.
    \end{align*}
    Now we can apply \cref{lem:energy_solution} to estimate
    \begin{align*}
        \|\Delta \bar\by\|_{\Lb^2(\Omega)} \leq
        k^2 \, \alphaMax\, 
        \left( L_\varphi \, 2\,C_0 + C_\varphi
        \right)  \,
        I_\varphi \, C_{\Omega}^2 \, {\|\balpha^1-\balpha^2\|}_\infty.
    \end{align*}
    The statement follows by setting $C \coloneqq \alphaMax\, 
        \left( L_\varphi \, 2\,C_0 + C_\varphi
        \right)  \,
        I_\varphi \, C_{\Omega}^2$.
\end{proof}

\subsection{Wellposedness of inverse-map}

This section is devoted to the stability analysis of the parameters with respect to small perturbations in the solution.
Let us again fix the control $\beps\in\Ead$. 
We are interested in the stability estimate of the form
\begin{align*}
	{\|\balpha^1-\balpha^2\|}_\infty \leq C_{\mathrm{inv}}\,{\|\by^{\balpha^1}-\by^{\balpha^2}\|}_{\sY}.
\end{align*}
This inequality implies the Lipschitz continuity of the inverse parameter-to-state map $\sP^{-1}_\beps$, i.e., the state-to-parameter map given for fixed controls $\beps \in \Ead$ as
\begin{align*}
    \sP^{-1}_\beps:\sB \to \sAK,
    \quad\by^\balpha \mapsto \balpha = \sP^{-1}_\beps(\by^\balpha),
\end{align*}
where we introduced the compact set $\mathscr B \coloneqq \sP(\sAK)\subset\sY$.
Therefore, if the difference of two solutions is small in the $\sY$ norm, then the corresponding parameters will be close with respect to the maximum norm (which is equivalent to all other norms in $\Rb^K$). 
This is the content of \cref{thm:stability_inverse_orig}.
Notice that an essential assumption is that the basis elements are bounded away from zero.
\begin{theorem}
	\label{thm:stability_inverse_orig}
	Suppose that \cref{asmpt:basis_elements} and  $k\in[K]$ hold.
	Furthermore, let $\balpha^1,\balpha^2 \in \sAk$.
    Then there exists a constant $C_{\mathrm{inv}}>0$ independent of $\balpha^1,\balpha^2$ such that
	\begin{align*}
		{\|\balpha^1-\balpha^2\|}_\infty\leq C_{\mathrm{inv}}\,{\|\by^{\balpha^1}-\by^{\balpha^2}\|}_{\sY}.
	\end{align*}
	Moreover, the constant has the form
	\begin{align*}
		C_{\mathrm{inv}} \coloneqq \sqrt{2 \widetilde{c}} \,   (1+ \,k  \, L_\varphi \, \alphaMax \, C_{\Omega})c_\varphi^{-1}
	\end{align*}
    for constants $\widetilde{c},C_\Omega>0$.
\end{theorem}

\begin{proof}
	With $\bar\balpha\coloneqq\balpha^1-\balpha^2$ and $\bPhi = (\bvarphi_1,\ldots,\bvarphi_K)$, we find
	\begin{subequations}
		\begin{align}
			&{\|\bar\balpha\|}_{\infty}^2\leq \widetilde{c} \, {\| \bar{\balpha}\|}_2^2
			= \widetilde{c} \, c_\varphi^{-2}\,{\| c_\varphi \, \bar\balpha \|}_2^2
			\leq \widetilde{c} \, c_\varphi^{-2} \, \left\|  \bPhi(\by^{\balpha_1}) \bar\balpha  + \bPhi(\by^{\balpha_2})\balpha_2 -\bPhi(\by^{\balpha_2})\balpha_2 \right\|_{\Lb^2(\Omega)}^2
		      \label{eq:equiv_norms_positivity}
		\\
			&\quad\leq 2\,\widetilde{c} \,c_\varphi^{-2} \left(
				{\| \bPhi(\by^{\balpha_1})\balpha_1
				- \bPhi(\by^{\balpha_2})\balpha_2 \|}_{\Lb^2(\Omega)}^2
				+
				{\| \left( \bPhi (\by^{\balpha_2})
				- \bPhi(\by^{\balpha_1})\right)\balpha_2 \|}_{\Lb^2(\Omega)}^2
			\right)
		\label{eq:triangle_square}
		\\
			&\quad\leq 2 \, \widetilde{c} \, c_\varphi^{-2} \left( \| \Delta (\by^{\balpha^1} - \by^{\balpha^2}) \|_{\Lb^2(\Omega)}^2 + k^2\,L_\varphi^2 \, \alpha_{\max}^2 \, \|\by^{\balpha^1} - \by^{\balpha^2} \|_{\Lb^2(\Omega)}^2 \right)
		\label{eq:solution_linearity_lipschitz}
		\\
			&\quad\leq 2\,\widetilde{c}\,c_\varphi^{-2} \, \big(1
            +k\,L_\varphi\,\alphaMax \big)^2 \|\by^{\balpha^1}-\by^{\balpha^2}\|_{\sY}^2.
		\label{eq:H2_regularity}
		\end{align}
	\end{subequations}
	In \eqref{eq:equiv_norms_positivity} we use the equivalence of norms in $\Rb^K$ and \cref{asmpt:basis_elements}-\ref{item:Basis_elements_bounded};
	in \eqref{eq:triangle_square} we apply the triangle inequality and a binomial inequality;
	in \eqref{eq:solution_linearity_lipschitz} we use that the difference $\by^{\balpha^1}- \by^{\balpha^2}$ fulfills \eqref{eq:weak_formulation_difference} and that $\bvarphi_j$ are Lipschitz continuous for $j \in [K]$;
	in \eqref{eq:H2_regularity} we exploit that the Laplace operator is an isometry.
\end{proof}

\section{Greedy-reconstruction algorithm}
\label{sec:GreedyAlgorithm}

In this section, we present our strategy to tackle the identification problem \eqref{eq:final_identification}.
For this, we recall some notation.
We denote by $\by^{\bbeta,\beps}$ the solution of \eqref{eq:Laplace_Model} with the control $\beps$ and nonlinearity $\bg_\bbeta$.
We have for $1\leq k \leq K-1$ the set
$
    \sAk \coloneqq \{ \balpha \in \sAK \, | \, \alpha_j = 0, \; j = k+1,\ldots K  \}
$ (cf. \eqref{eq:set_Ak}).
Further, we set $\be^{k} \in \Rb^K$ as the $k$-th canonical vector in $\Rb^K$.
Hence, $\by^{\be^k,\beps}$ denotes the solution of \eqref{eq:Laplace_Model} with the control $\beps$ and nonlinearity $\bvarphi_k$.

We show now how to construct a set of optimized controls and basis elements, in particular with the goal of improving local convexity properties of \eqref{eq:final_identification}.
We follow the idea of \cite{BuchwaldCiaramella2021GreedyReconstruction,BuchwaldCiaramella2021ReconstructionofA,BuchwaldCiaramella2021OGR_SpinIdentification,MadaySalomon2009GreedyQuantum}, where linear and bilinear dynamical systems have been investigated.
It is the aim of this work to extend this strategy to the general semilinear PDE case \eqref{eq:Laplace_Model}.

The main idea is to split the reconstruction process of $\bg_{\balpha^\star}$ into offline and online phases.
In the offline phase, a greedy algorithm computes a set of optimized controls $\{\beps^m\}_{m=1}^K$ by exploiting only simulations of the semilinear model \eqref{eq:Laplace_Model} and without using any laboratory data. 
In the online phase, the computed controls $\{\beps^m\}_{m=1}^K$ are used experimentally to produce the laboratory data $\by_\star^{\beps^m}$ for $m \in [L]$ and to define the nonlinear problem \eqref{eq:Laplace_Model}.
While the online phase consists (mathematically) of solving a classical parameter-identification inverse problem, the offline phase requires the greedy algorithm that was first introduced in \cite{MadaySalomon2009GreedyQuantum} and analyzed and improved in \cite{BuchwaldCiaramella2021ReconstructionofA}.
The goal of this offline/online framework is to find a good approximation of the unknown operator $\bg_{\balpha^\star}$ for which the difference between observed experimental data and numerically computed data is the smallest for any control.
To do so, the algorithm attempts to distinguish numerical data for any two $\bg_{\balpha^1}$, $\bg_{\balpha^2} \in \text{span } \mathcal{G}$ \cite{MadaySalomon2009GreedyQuantum}.
This is achieved by an iterative procedure that performs a sweep over the basis $\mathcal{G}$ and computes a new control at each iteration.
Suppose that at iteration $k$ the control fields $\beps^1,\ldots,\beps^k$ are already computed; the new control function $\beps^{k+1}$ is obtained by two substeps, the so-called fitting and splitting step.
One first solves the identification problem \eqref{eq:greedy_fitting_step} which gives the coefficients $\bbeta^k=(\beta^k_j)_{j=1,\dots,k}$.
Then one computes the new control as the solution of the splitting step \eqref{eq:greedy_discriminatory_step}.

This general procedure can be optimized using some extensions.
More specifically, in each iteration, all elements of $\mathcal{G}$ are considered in parallel and the `best' one is chosen in a greedy way; see \cite[Section 6.1]{BuchwaldCiaramella2021GreedyReconstruction} for further explanation of the original and optimized greedy reconstruction algorithm in a linear case.

This offline procedure for our nonlinear case is summarized in \cref{algo:ONGR}.
We refer to \cite{BuchwaldCiaramella2021ReconstructionofA} for a general idea to prove that this algorithm is able to make problem \eqref{eq:final_identification} (locally) uniquely solvable.

Notice, that during the Algorithm, the basis $\mathcal{G}$ is reordered. Since only $k$ controls are generated through the execution of the algorithm, this has the effect that actually only the $k$ most sensible basis elements are selected, in the sense that they lead to a high difference in the splitting step, in particular higher than the tolerance $\mathrm{tol}_1$ (cf. \eqref{eq:greedy_discriminatory_step}).
In the final identification, only the first $k$ basis elements are taken into account. Hence, the reordering has an influence if less than $K$ controls are selected.
In general, the value of $\mathrm{tol}_1$ allows the algorithm to stop with a smaller amount of optimized controls, if any further controls do not provide additional information about the nonlinearity within the given tolerance.
In this work, we focus on the best possible reconstruction ability of our algorithm given a certain set $\mathcal{G}$. 
Hence, we keep $\mathrm{tol}_1$ fixed at machine precision and, therefore, generate the maximum amount of optimized controls.

\begin{algorithm}
	\caption{Optimized Nonlinear Greedy Reconstruction (ONGR) Algorithm}
	\label{algo:ONGR}
	\begin{small}
		\begin{algorithmic}[1]
			\Require A set of $K$ basis functions $\mathcal{G}=\{\bvarphi_1,\ldots,\bvarphi_K\}$
			and a tolerance $\mathrm{tol}_1>0$, and a regularization parameter $\nu>0$
			\State Compute $\beps^1$ and the index $i_1\in[K]$ by solving the \emph{initialization} problem 
			\begin{align*}
				\max_{i\in[K]}\max_{\beps \in \Ead} 
				\frac{1}{2}\standardNorm{\by^{\boldsymbol{0},\beps}-\by^{\be^i,\beps}}_{\Lb^2(\Omega)}^2 +\frac{\nu}{2} {\|\beps\|}_{\Lb^2(\Omega)}^2.
			\end{align*}
			\State Swap $\bvarphi_1$ and $\bvarphi_{i_1}$ in $\mathcal{G}$, and set $k=1$ and 
            $f_{max}=\frac{1}{2}\,\big\|\by^{\boldsymbol 0,\beps^1}-\by^{\be^{i_1},\beps^1}\big\|_{\Lb^2(\Omega)}^2$.
			\While{ $k\leq K-1$ and $f_{max}>\textrm{tol}_1$}
			\For{$\ell=k+1,\ldots,K$}
			\State \emph{Fitting step}: 
            Find $\bbeta^\ell$ that solves the problem
			\begin{equation}
				\label{eq:greedy_fitting_step}
				\min_{\bbeta\in\sAk}\frac{1}{2}
				\sum_{m=1}^{k}
				\standardNorm{\by^{\bbeta,\beps^m}-\by^{\be^{\ell},\beps^m}}_{\Lb^2(\Omega)}^2.
			\end{equation}
			\EndFor
			\State \emph{Splitting step}: Find $\beps^{k+1}$ and $i_{k+1}$ that solve the problem
			\begin{equation}
                \label{eq:greedy_discriminatory_step}
				\max_{i\in\{k+1,\ldots,K \}}
				\max_{\beps \in\Ead}\frac{1}{2}
				\standardNorm{\by^{\bbeta^{i},\beps}-\by^{\be^i,\beps}}_{\Lb^2(\Omega)}^2
                +\frac{\nu}{2}\,{\|\beps\|}_{\Lb^2(\Omega)}^2.
			\end{equation}
			\State Swap $\bvarphi_{k+1}$ and $\bvarphi_{i_{k+1}}$ in $\mathcal{G}$.
            \State Set 
            $f_{max}=\frac{1}{2}
				\|\by^{\bbeta^{i_{k+1}},\beps^{k+1}}-\by^{\be^{i_{k+1}},\beps^{k+1}}\|_{\Lb^2(\Omega)}^2$.
            \State Set $k \leftarrow k+1$.
			\EndWhile\\
			\Return $\mathcal{G}$ and the computed $\{\beps^m\}_{m=1}^k$.
		\end{algorithmic}
	\end{small}
\end{algorithm}


We now present the variational inequality corresponding to the first-order necessary optimality condition for the optimization problems appearing in \cref{algo:ONGR}.
For this, we in particular derive the adjoint equation. Notice that the variational inequality is equivalent to the Karush-Kuhn-Tucker system; see, e.g. \cite[Section 1.4]{Treoltzsch2010OCP_PDE}.
Let us define  $\bg_k(y)\coloneqq \bvarphi_k(y)$ for $y=(y_1,y_2)^\top \in\Rb^2$ an $k\in [K]$.
We start with the discussion of the fitting step at the $k$-th iteration. 
For this let us introduce $\sY^k \coloneqq \sY \times \ldots \times \sY$ ($k$ times) and define $\{\by^m_{k+1}\}_{m=1}^k \subset \sY^k$ as the solutions to
\begin{align*}
	-\Delta \by^m_{k+1} + \bg_{k+1}(\by^m_{k+1}) 
    = c\,\beps^m
	\qquad
	\text{ for }m \in [k]
\end{align*}
with homogeneous Dirichlet boundary conditions.
For the fixed controls $\{\beps^m\}_{m=1}^k \subset\Ead$ and with given $\{\by^m_{k+1}\}_{m=1}^k \subset \sY^k$,
the fitting step is given as the following minimization problem (the subscript $f$ stands for \enquote{fitting step})
%
\begin{align}
    \begin{split}
    	&\min J_f(\bbeta, \by^1,\ldots,\by^k)\coloneqq\sum_{m=1}^k\frac{1}{2} \standardNorm{\by^m - \by^m_{k+1}}_{\Lb^2(\Omega)}^2
        \\
    	&\hspace{10mm}\text{subject to }  -\Delta \by^{m} + \bg_{\bbeta}(\by^{m}) = \beps^m,
        \quad \by^m_{|\Gamma} = 0
        \quad
        \text{ for }m \in [k]
    	\\
    	&\hspace{10mm}\text{and }  \bbeta \in \sAk.
    \end{split}
    \label{eq:fitting_step_problem}
\end{align}
Using the parameter-to-state map $\sP_{\beps}$ defined in \cref{prop:paramter-to-state}, we can introduce the reduced cost 
\begin{align*}
	\widehat{J}_f(\bbeta) \coloneqq J_f(\bbeta, \sP_{\beps^1}(\bbeta),\ldots,\sP_{\beps^k}(\bbeta))
\end{align*}
and the reduced problem
\begin{align}
    \min\widehat J_f(\bbeta)
    \quad
    \text{s.t.}
    \quad
    \bbeta\in\sAk.
    \label{eq:fitting_step_reduced}
\end{align}
We can state the following theorem.
\begin{theorem}
    \label{thm:existence_solution_fitting_step}
    Let \cref{asmpt:basis_elements,asmpt:monotonicity_nonlinearity} hold.
    Then the optimization problem \eqref{eq:fitting_step_problem} has at least one solution $\bar\bbeta \in \sAk$.
\end{theorem}
\begin{proof} 
    Notice that the reduced cost $\widehat{J}_f$ is non-negative and thus there exists a minimizing sequence $\{\bbeta^\ell\}_{\ell \in \Nb} \subset \sAk$ such that
    \begin{align*}
        0\le\bar J_f = \inf \{ \widehat{J}_f(\bbeta) \, | \, \bbeta \in \sAk \} = 
        \lim_{\ell \rightarrow \infty} \widehat{J}_f(\bbeta^\ell) < \infty.
    \end{align*}
    We set $\by^{\bbeta,\beps}\coloneqq\{\sP_{\beps^m}(\bbeta)\}_{m=1}^k \subset \sY^k$.
    Due to \cref{thm:Energy_laplace_solution}, the sequence $\{\by^{\bbeta^\ell,\beps}\}_{\ell \in \Nb} \subset \sY^k$ is well defined.
    Since $\sAk$ is compact in $\mathbb{R}^k$, there exists a subsequence (which is still labeled by $\bbeta^\ell$) $\{\bbeta^\ell\}_{\ell\in\Nb} \subset \sAk$ of the minimizing sequence and an element $\bar\bbeta \in \sAk$ such that we have the strong convergence $\bbeta^\ell \rightarrow \bar\bbeta$ in $\Rb^k$ as $\ell \rightarrow \infty$.
    Due to \cref{lem:energy_solution} we have $\by^{\bbeta^\ell,\beps}\rightarrow\bar\by \coloneqq \by^{\bar\bbeta,\beps}$ in $\sY^k$.
    Using the lower semicontinuity of the norms and the continuity of the parameter-to-state map from $\sAK$ to $\sY$ (\cref{lem:energy_solution})
    \begin{align*}
        \widehat{J}_f(\bar\bbeta) \leq \lim_{\ell \rightarrow \infty} \widehat{J}_f(\bbeta^\ell) = \bar J_f.
    \end{align*}
\end{proof}
Let us now introduce the adjoint variable $\bq = \bq(x) = (q_1(x),q_2(x))^\top$. 
The adjoint equations for the fitting step read as follows
\begin{align*}
    - \Delta \bq^{m}
    + \bg'_{\bbeta}(\by^{m})^\top\bq^{m}
    =- \left(\by^m - \by^m_{k+1}\right),
    \qquad
    \text{for } m \in [k],
\end{align*}
completed with homogeneous Dirichlet boundary conditions. Since $\nabla\bg_\bbeta(y)$ is positive semidefinite, we can ensure a constrained qualification for \eqref{eq:fitting_step_reduced}, so that first-order necessary optimality conditions can be formulated; cf. \cite[Section~6.1.2]{Treoltzsch2010OCP_PDE}.

Supposing that $\bar\bbeta$ solves \eqref{eq:fitting_step_reduced}, the optimality condition of the fitting step in iteration $k$ reads
\begin{align*}	
	\left\langle 
        \nabla \widehat{J}_f(\bar\bbeta),\bbeta - \bar\bbeta
     \right\rangle_2 \geq 0
     \qquad\qquad
     \forall  \bbeta \in \mathscr{A}_{k},
     \qquad
\end{align*}
where $\nabla \widehat{J}_f = (\nabla_{\beta_1} \widehat{J}_f,\ldots,\nabla_{\beta_k} \widehat{J}_f)^\top$ is given by
\begin{align*}
     \nabla_{\beta_j}\widehat{J}_f = \nu \, \beta_j + 
    	\sum_{m=1}^k
    	\int_{\Omega} \bq^{m}(x) \cdot \bvarphi_{j} (\by^{m}(x))\rmd x,
        \qquad\qquad
        j \in [k].
\end{align*}
Now we fix $\bbeta$ and turn to the splitting step in which the new control $\beps^{k+1}$ is computed. 
The splitting step at iteration $k<K$ is given by
(the subscript $s$ stands for \enquote{splitting step})
\begin{align*}
    \begin{split}
    	&\min
    	J_s(\beps,\by^\bbeta,\by^{k+1})
        \coloneqq -\frac{1}{2} \standardNorm{\by^\bbeta - \by^{k+1}}_{\Lb^2(\Omega)}^2 + \frac{\nu}{2}\standardNorm{\beps}^2_{\Lb^2(\Omega)}
    	\\
    	&\hspace{5mm}\text{subject to} \quad
    	\begin{cases}
    		-\Delta \by^{\bbeta} + \bg_{\bbeta}(\by^\bbeta) = \beps, \quad \by^{\bbeta}_{|\Gamma}=0, 
        \\
    		-\Delta \by^{k+1} + \bg_{k+1}(\by^{k+1}) = \beps, \quad \by^{k+1}_{|\Gamma}=0,
    	\end{cases}
	    \\
	    &\hspace{5mm}
    	\text{and } \beps \in \Ead.
    \end{split}
\end{align*}
Exploiting the control-to-state map $\sS$ defined in \cref{thm:Continuity_control-to-state-map}, we can introduce the reduced cost
\begin{align*}
	\widehat{J}_s(\beps) \coloneqq 
    J_s(\beps,\sS_\bbeta(\beps), \sS_{\be^{k+1}}(\beps)),
\end{align*}
where $\be^{k} \in \Rb^K$ denotes again the $k$-th unit vector in $\Rb^K$.
We also introduce the reduced problem
\begin{align}
    \min\widehat J_s(\beps)
    \quad
    \text{s.t.}
    \quad
    \beps\in\Ead.
    \label{eq:splitting_step_reduced}
\end{align}
We can state the following theorem on the existence of solutions to \eqref{eq:splitting_step_reduced}.

\begin{theorem}
    \label{thm:existence_solution_splitting_step}
    Let \cref{asmpt:basis_elements,asmpt:monotonicity_nonlinearity} hold.
    Then the optimization problem \eqref{eq:splitting_step_reduced} has at least one solution $\bar\beps \in \Ead$.
\end{theorem}
\begin{proof}
    Since our $\widehat J_s$ is bounded from below by the virtue of \cref{thm:Energy_laplace_solution} and the boundedness of the set of admissible controls $\Ead$, we can apply \cite[Theorem 3.1]{CasasRoesch2020NonmonotoneSemilinearElliptic}.
\end{proof}

The adjoint equations for the fitting step read as follows
\begin{align*}
	-\Delta\bq^{\bbeta}  + \bg'_{\bbeta}(\by^{\bbeta})^\top \bq^{\bbeta}
	&=- \nabla_{\by^{\bbeta}} J_s(\beps;\by),
	\\
	-\Delta\bq^{k+1} + \bg'_{k+1}(\by^{k+1})^\top \bq^{k+1}
	&=- \nabla_{\by^{k+1}} J_s(\beps;\by),
\end{align*}
completed with homogeneous Dirichlet boundary conditions.
Since $\nabla\bg_\bbeta(y)$ and $\nabla\bg_{k+1}(y)$ are positive semidefinite, we can ensure a constrained qualification for \eqref{eq:fitting_step_reduced}, so that first-order necessary optimality conditions can be formulated; cf. \cite[Section~6.1.2]{Treoltzsch2010OCP_PDE}.

Suppose that $\bar\beps$ solves \eqref{eq:splitting_step_reduced}, then an optimality condition of the splitting step reads
\begin{align*}
	\left\langle \nabla_\beps \widehat{J}_s(\bar\beps), \beps - \bar\beps \right\rangle_{\Lb^2(\Omega)} \geq 0
    \qquad\qquad
    \forall  \beps \in \Ead,
\end{align*}
where $\nabla_\beps \widehat{J}_s$ is given $\nabla_\beps \widehat{J}_s(\beps) = \nu\,\beps - (\bq^\bbeta + \bq^{k+1})$.

We solve all optimization problems in \cref{algo:ONGR} using the MATLAB built-in function \texttt{fmincon} which uses the interior-point method; see, e.g., \cite{NocedalWright2006NumOpt}. However, any other suitable subroutine could be considered.

\section{Numerical approximation}
\label{sec:NumApprox}

In this section, we present the numerical approximation of the governing equation.
First, we describe our method to solve the governing model \eqref{eq:Laplace_Model}. 
We choose now $\Omega=(-x_{\max},x_{\max})^2$ for a fixed $x_{\max}>0$. We set a numerical grid that provides a partitioning of $\Omega$ in $N\times N$, $N > 1$, equally-spaced non-overlapping square cells of side length $h = 2x_{\max}/N$. 
We define the nodal points
\begin{align*}
	x_1^i=i\,h- x_{\max},\qquad x_2^j = j \, h - x_{\max}.
\end{align*}
Our discrete domain is then given by
\begin{align*}
	\Omega_{h}\coloneqq \bigcup_{i,j = 1}^{N} \omega_{ij},
\end{align*}
where an elementary cell is defined as
\begin{align*}
	\omega_{ij} \coloneqq \left\lbrace (x_1,x_2) \in \Omega \; \big\vert \;
	(x_1,x_2) \in 
	\left(x_1^{i-1}, x_1^{i}\right) \times \left(x_2^{j-1}, x_2^j\right)  \right\rbrace,\quad i,j \in [N].
\end{align*}
In our numerical examples, we use $x_{\max}=1$.

We define the discrete vectors $\by^N,\beps^N \in\Rb^{2(N+1)^2}$ as $\by^N = (y^0_1,\ldots,y^{N^2}_1,y_2^0,\ldots y^{N^2}_2)^\top$ and $\beps^N = (\varepsilon^0_1,\ldots,\varepsilon_1^{N^2},\varepsilon^0_2,\ldots,\varepsilon_2^{N^2})^\top$.

After discretization of \eqref{eq:Laplace_Model}, we end up with 
\begin{align}
	\label{eq:Laplace_Model_discrete}
	-A \by^N + \bg(\by^N) = \beps^N,
\end{align}
where $A$ is the discrete Laplace operator using finite differences and zero boundary conditions \cite[Section 2.2]{JovanovicSueli14}.

We consider a fixed-point method for solving \eqref{eq:Laplace_Model_discrete} numerically \cite{Amann1975NonlinearFixedPoint,CiaramellaGander2018LectureIterative}.
The procedure with relaxation parameter $\lambda_{\mathrm{a}} \in [0,1]$ is summarized in \cref{algo:fix-point}.
\begin{algorithm}
	\caption{Fixed-point method}
	\label{algo:fix-point}
	\begin{algorithmic}[1]
		\Require Parameter $0\leq\lambda_{\mathsf a}\leq 1$, 
  tolerance $\mathrm{tol}_2>0$, maximum iteration depth $\ell_{\max}\gg1$
		\State Set $\ell=0$ and $E\gg \mathrm{tol}_2$
		\State Generate initial guess $\by_0^N$ that satisfies $-A\by^N_0 = \beps^N$
		\While{$E<tol$ \textbf{and} $\ell<\ell_{\max} $}
		      \State Calculate $\widetilde{\by}^N$ as solution to
		      $-A\widetilde{\by}^N = \beps^N - \bg(\by^N_{\ell})$
		      \State Set $\by^N_{\ell+1}=\lambda_{\mathsf a}\by^N_\ell+(1-\lambda_{\mathsf a})\widetilde{\by}^N$ and $E= h \, \|\by^N_{\ell+1}-\by^N_{\ell}\|_2$
		\State Set $\ell \leftarrow \ell+1$
		\EndWhile \\
		\Return $\by^N_\ell$
	\end{algorithmic}
	\label{GlNVAlg}
\end{algorithm}
Next, we explain the ingredients for \cref{algo:ONGR} that we use in our numerical experiment. We choose the following nonlinearity, motivated by coupled Lotka-Volterra equations in the stationary regime (see, e.g., \cite{GambinoLombardoSammartino2009LotkaVolterraPDEUncontrolledNonlinear,Ibanez2017OC_LotkaVolterra,MortujaChaube2021AnalysisPredatorPraySystem}).
In this setting, the variables $y_1$ and $y_2$ have the significance of the distribution of two different species that interact with the nonlinearity $\bg_\balpha$. 
Furthermore, usually this interaction is up to constants equal in both components. 
Hence, we choose
\begin{align*}
    \bg_\balpha(y) \coloneqq (\gamma_1 \, G_\balpha(y), -\gamma_2 \, G_\balpha(y))^\top,
    \label{eq:LotkaVolterra_nonlin}
\end{align*}
where $\gamma_1 \geq \gamma_2>0$ are given constants and
\begin{align}
    G_\balpha(y) \coloneqq \sum_{j=1}^K \alpha_j \varphi_j(y),
\end{align}
with basis functions $\varphi_j : \Rb^2\rightarrow\Rb$, $j \in [K]$.

We use the two-dimensional polynomials as basis functions to reconstruct the unknown nonlinearity $G_\balpha$. 
More precisely, we approximate the nonlinearity using monomials up to entrywise degree $P$, i.e. the basis
\begin{align}
    \mathcal{G}
    =\big\{\varphi_1,\ldots,\varphi_{K}\big\}
    \coloneqq \big\{1,y_1,y_2,y_1y_2,y_1^2,y_2^2,y_1^2y_2,\ldots,y_1^Py_2^P\big\}.
	\label{eq:Polynomials_definition}
\end{align}
The cardinality of $\mathcal{G}$, which is the maximum number of iterations $K$ in the algorithm, is then given by $K=|\mathcal{G}| = \binom{2+P}{P}$.
With this setting, the support of the solution $\by^\balpha$ of \eqref{eq:Laplace_Model} does not need to be known in advance. 
Furthermore, it is possible to make predictions of the nonlinearity outside of the support of the solutions due to the non-locality of the monomials.

We now discuss \cref{asmpt:basis_elements} for our choice of the basis elements:
1) All basis elements $\varphi_j$ are smooth and in particular (locally) Lipschitz continuous since they are monomials; 
2) The basis elements $\varphi_j$ are not bounded away from zero and in general not bounded from above.
However, on every bounded set in $\Rb^2_+$ they are also bounded and strictly greater than zero.
3) All monomials $\varphi_j$ are monotonically increasing for nonnegative entries. 
However, by our special structure \eqref{eq:LotkaVolterra_nonlin}, the nonlinearity $\bg_\balpha$ is not monotonically increasing on the full space $\Rb^2$ but on a (non-empty) subspace.
   
Hence our choice of monomials does not fit perfectly with our assumptions.
However, they can be fulfilled after restricting ourselves to a non-empty subset of $\Rb^2$. Furthermore, even though the theoretical assumptions are not completely fulfilled, we observe that our strategy works very well in our numerical experiments.

\section{Numerical experiments}
\label{sec:NumExperiments}

In the following, we present the results of numerical tests in order to validate the reconstruction ability of our strategy described above.

We consider three different desired nonlinearities $G_\star^i$, $i=1,2,3$ and aim at reconstructing them. 
All have different fundamental properties:
One of them lies within the span of the basis $\mathfrak{G}$ (cf. \cref{sec:Setting}).
The other two can only be approximated by monomials. One of them is bounded and enjoys some similarities to the basis elements and can therefore be approximated up to some tolerance more easily by the basis elements of $\mathcal{G}$ while the other one is unbounded more basis elements are needed to approximate it to the same tolerance.
Specifically, we choose
\begin{align}
    \label{eq:Nonlinearities}
	G^1_\star(y)= 0.05y_1y_2,
	\quad 
    G^2_\star(y) = 0.01\sin(2y_1)\sin(2y_2),
    \quad
    G^3_\star(y)= 0.01\exp(2y_1)\exp(2y_2).
\end{align}
These nonlinear functions are plotted in \cref{fig:Desired_nonlinearites}.
\begin{figure}
	\begin{subfigure}[l]{\thirdsPlotsfactor\textwidth}
		\includegraphics[width=\textwidth]{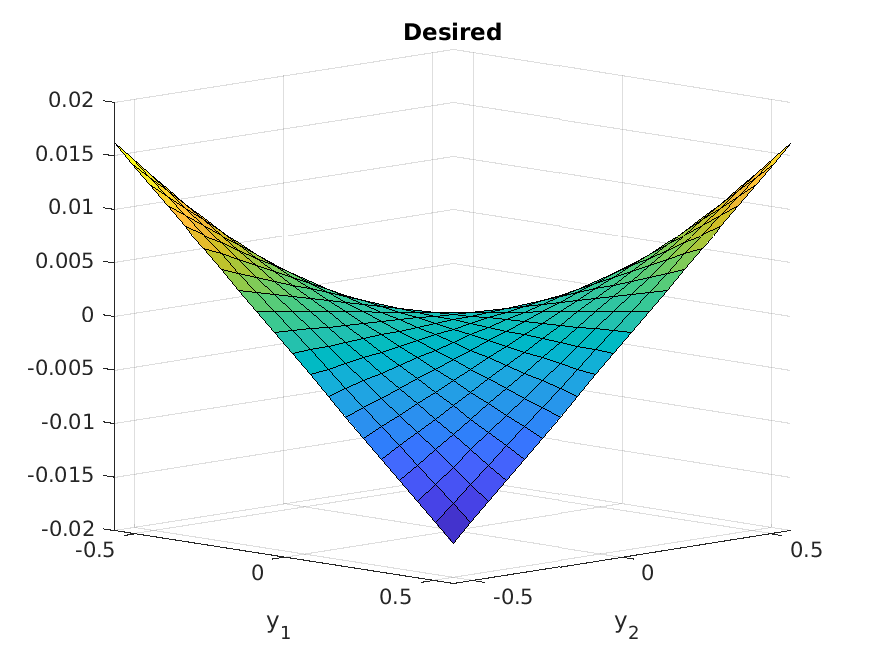}
		\caption{Desired bilinear $G^1_\star$}
		\label{subfig:Bilinear}
	\end{subfigure}
	\hfill
	\begin{subfigure}[l]{\thirdsPlotsfactor\textwidth}
		\includegraphics[width=\textwidth]{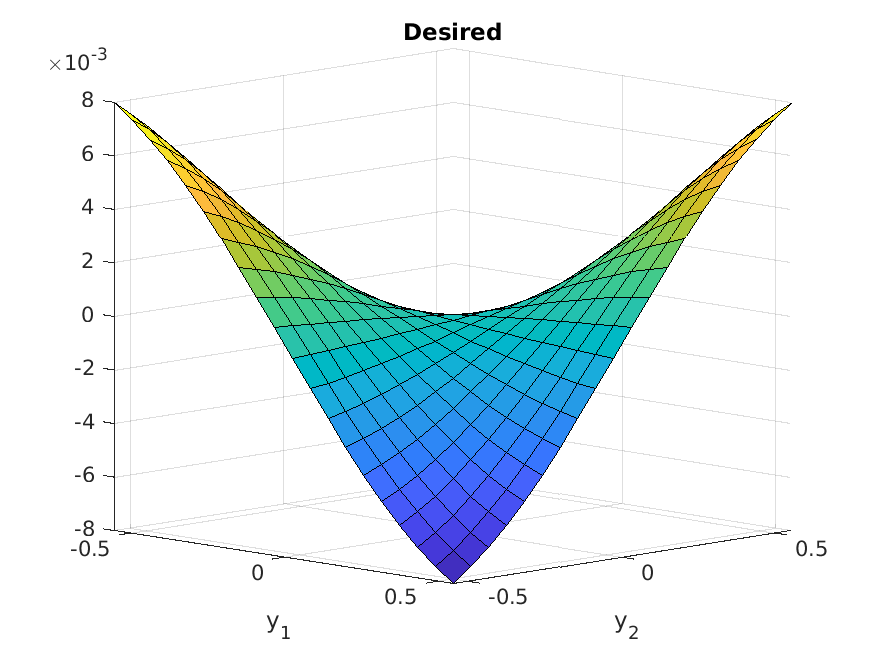}
		\caption{Desired sinusoidal $G^2_\star$}
		\label{subfig:sinusoidal}
	\end{subfigure}
	\hfill
	\begin{subfigure}[l]{\thirdsPlotsfactor\textwidth}
		\includegraphics[width=\textwidth]{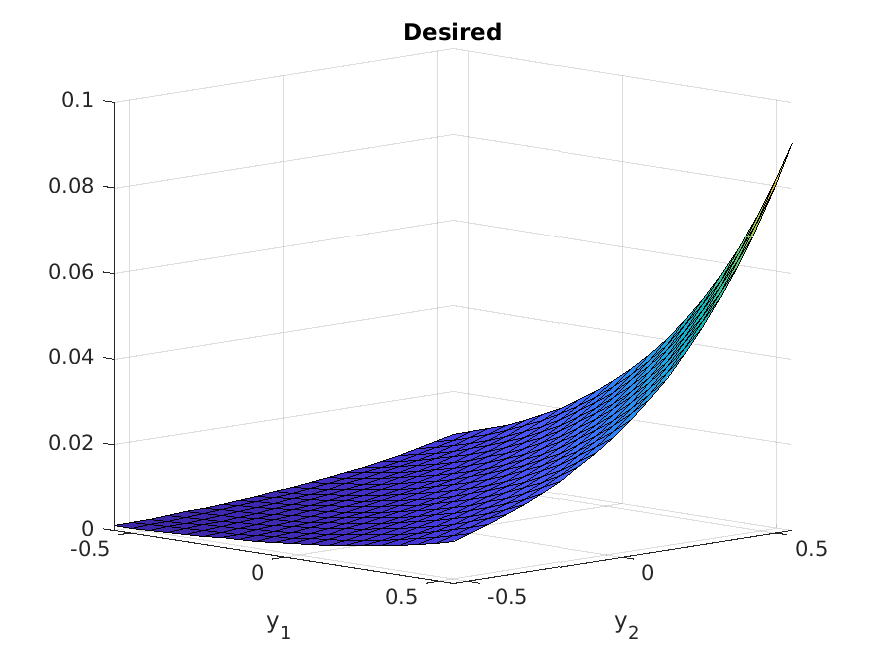}
		\caption{Desired exponential $G^3_\star$}
		\label{subfig:exponential}
	\end{subfigure}
	\caption{The three different desired nonlinearities $G^i_\star$, $i=1,2,3$ defined in \eqref{eq:Nonlinearities}.}
	\label{fig:Desired_nonlinearites}
\end{figure}
Furthermore, we use $\beps_\mathsf a = (-1,-1)^\top$, $\beps_\mathsf b = (1,1)^\top$, $\gamma_1=\gamma_2=0.2$ and run \cref{algo:ONGR} for three different bases $\mathfrak{G}$ for monomials of (entrywise) degree $P = 2,3,5$, respectively.
In our setting, we use $\lambda_{\mathsf a}=0$ in \cref{algo:fix-point}.

We now introduce the \emph{solution set} $\Sc^m_\balpha$ corresponding to the $m$-th control and using the coefficients $\balpha\in\sAK$ as
	\begin{align}
		\Sc^m_\balpha \coloneqq \left\lbrace \by^{\balpha,\beps^m}(x) \, \Big| \, x \in \Omega  \right\rbrace \subset \Rb^2.
		\label{eq:solutionSet}
	\end{align}
Furthermore, we define $\Sc_\balpha = \cup_{m=1}^K \Sc^m_\balpha$.

In order to evaluate our results, we consider the error
\begin{align}
    \mathfrak{e}(y;\balpha) &\coloneqq
    G_\star(y) - g_\balpha(y)
    = 
    G_\star(y) - \sum_{j=1}^K \alpha_j \varphi_{j} (y)
    \label{eq:definition_error}
\end{align}
between the true and reconstructed nonlinearity for $\balpha \in \sAK$ and $y \in \Omega_\Sc \subset \Rb^2$.

The subset $\Omega_\Sc$ is the smallest square in $\Rb^2$ that contains all points of the corresponding solution set $\Sc_\balpha$.
Notice that the final identification only takes the values on the solution set $\Sc_\balpha$ into account because only there we have information from the underlying system of equations.
However, since we know the nonlinearities that we try to reconstruct, we are able to define the error on the whole domain $\Omega_\Sc$.

Recall, that the true and reconstructed nonlinearities $G_\star$ and $G_\balpha$ are defined on the whole $\Rb^2$ since we use nonlocal basis functions. Because of this, the definition of the error $\mathfrak{e}$ can be extended from $\Omega_\Sc$ to $\Rb^2$ in a straight forward manner.

\subsection{Reconstruction of a bilinear function}
\label{sec:NumExperiments2}
Let us consider first the bilinear $G^1_\star(y)= 0.05y_1y_2$, shown \cref{subfig:Bilinear}, which clearly lies in the span $\mathcal{G}$ for any order $P\geq 1$.
We run \cref{algo:ONGR}, use the selected controls to compute the data, and solve the corresponding online identification problem \eqref{eq:final_identification}.
The results are shown in \cref{fig:bilinear_nonlinearity}, where the three different subfigures correspond to the three different bases $\mathcal{G}$ of order $P = 2,3,5$.
\begin{figure}
	\begin{subfigure}[l]{\thirdsPlotsfactor\textwidth}
		\includegraphics[width=\textwidth]{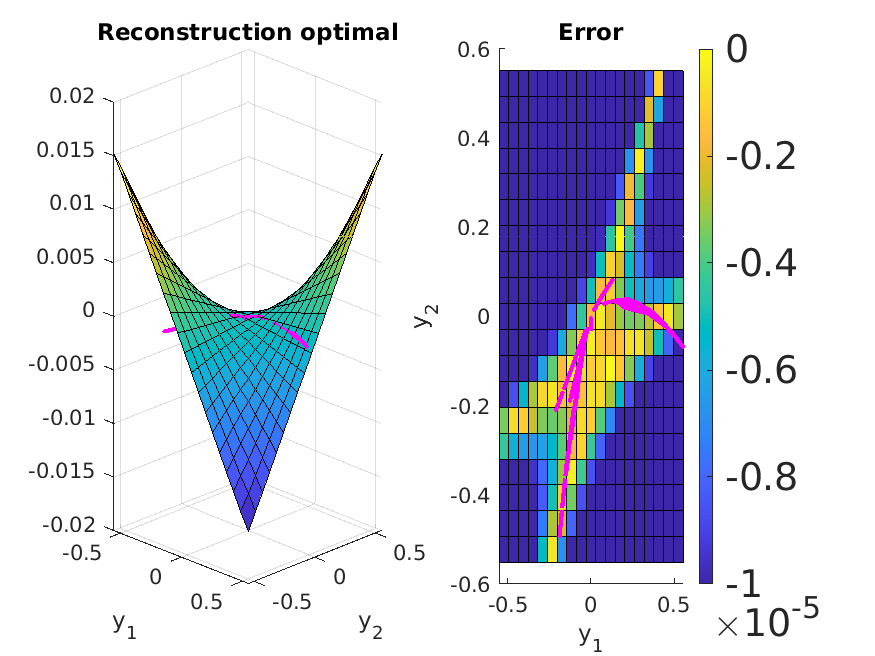}
	\end{subfigure}
	\hfill
	\begin{subfigure}[l]{\thirdsPlotsfactor\textwidth}
		\includegraphics[width=\textwidth]{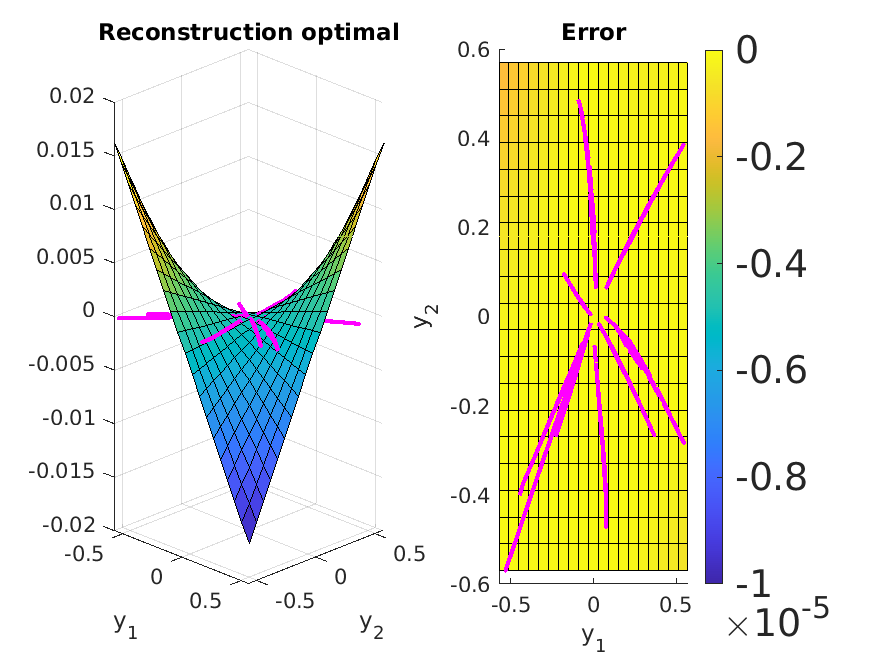}
	\end{subfigure}
	\hfill
	\begin{subfigure}[l]{\thirdsPlotsfactor\textwidth}
		\includegraphics[width=\textwidth]{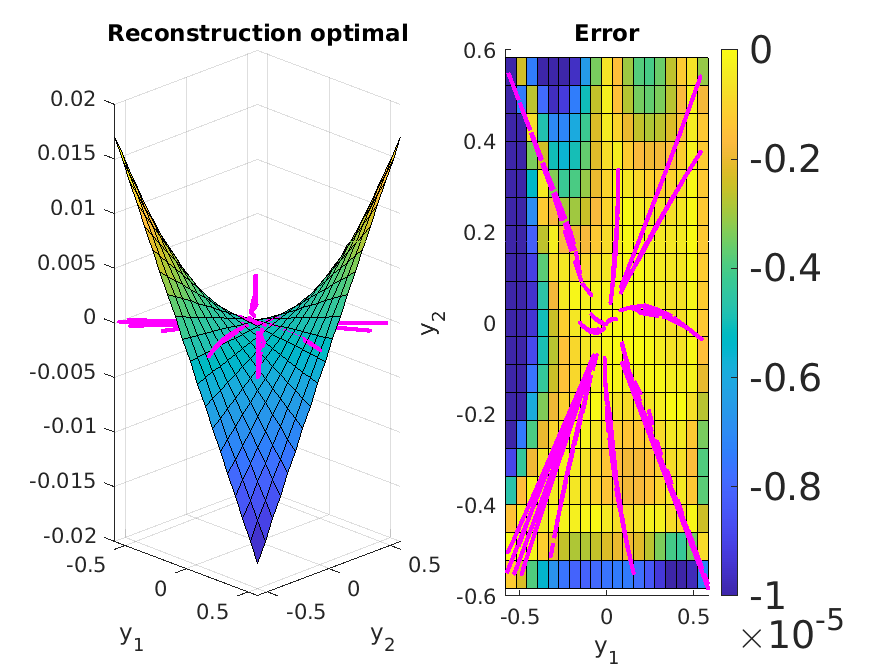}
	\end{subfigure}
	\caption{Reconstruction and error plots (cf. \eqref{eq:definition_error}) for $P=2,3,5$ and the bilinear nonlinearity with solutions curves in magenta.}
	\label{fig:bilinear_nonlinearity}
\end{figure}

Each subfigure of \cref{fig:bilinear_nonlinearity} contains the nonlinearity $G_{\balpha}$ reconstructed by problem \eqref{eq:final_identification} (left plot) and the error $\mathfrak{e}(y;\balpha)$ between the true and the reconstructed nonlinearity (right plot).
In the error plot, the solution sets $\Sc^m_\balpha$, $m \in [ K]$, are plotted in magenta.

We observe that the order of magnitude of the error is quite small for all three cases.
As discussed above, this is to be expected since the desired nonlinearity $G^1_\star$ is one of the basis elements.
However, even though the final identification for $P=5$ has access to the most data points (as shown by the number of magenta lines), the error is actually smallest for $P=3$.
The reason for this is that for $P=5$ more basis elements with higher polynomial degrees are taken into account.
Hence, in particular, at the boundary of the domain the error $\mathfrak{e}$ might become quite large if this part of the domain does not contain the solution sets; see also the explanation after the definition of the error in \eqref{eq:definition_error}.
To avoid this issue, it even seems that \cref{algo:ONGR} is trying to select the controls such that more parts of $\Omega_\Sc$ contain solution sets (cf. beginning of \cref{sec:NumExperiments}).
This is particularly apparent for $P=3$ and $P=5$.

In order to investigate further how \cref{algo:ONGR} chooses the controls, we consider the solutions
\begin{align*}
	y_1(x) = \eta \, \kappa(x_1,x_2),
	\quad
	y_2(x) = -\vartheta \, \kappa(x_1,x_2)
	\quad
	\text{ with }
	\;\kappa(x_1,x_2) \coloneqq \sin\left(\frac{x_1+1}{2}\pi\right) \sin\left(\frac{x_2+1}{2} \pi\right).
\end{align*}
inside the domain $\Omega=(-1,1)^2$, for some $\eta,\vartheta \in \Rb\backslash \{0\}$ with $\vartheta>\eta$.
Then, the graphs of the solutions $(y_1,y_2)$ in the $y_1$--$y_2$ plane are given by straight lines starting in the origin, having slope $\nicefrac{-\vartheta}{\eta}$ and length $\sqrt{\eta^2+\vartheta^2}$. 
Hence, we can reach every point in the $y_1$--$y_2$ plane within the $[\eta,\vartheta]$ rectangle.
The partial derivatives for $i=1,2$ are given by
\begin{align*}
	\partial^2_{x_ix_i} y_1(x)=-\frac{\pi^2}{4} y_1(x),
    \qquad
    \partial^2_{x_ix_i} y_2(x) =\frac{\pi^2}{4} y_2(x),
\end{align*}
which implies by the structure of our equation that 
\begin{align}
	\label{eq:constructed_controls}
	\begin{aligned}
		-\Delta y_1 + 0.05\gamma_1 y_1 y_2 
		&=\eta \frac{\pi^2}{2} \kappa(x_1,x_2)
		-0.05\gamma_1 \, \eta \, \vartheta \, \kappa(x_1,x_2)^2,
  \\
		-\Delta y_2 -0.05\gamma_2 y_1  y_2
		&=-\vartheta \frac{\pi^2}{2} \kappa(x_1,x_2)
		+0.05\gamma_2 \, \eta \, \vartheta \,\kappa(x_1,x_2)^2.
	\end{aligned}
\end{align}
Recalling \eqref{eq:Laplace_Model}, we observe that by choosing the controls according to the right-hand side of \eqref{eq:constructed_controls}, the corresponding solutions are able to reach any point within the $[\eta,\vartheta]$ rectangle.

In \cref{fig:results_controls}, we plot on the left a representative selection of the controls $\{\beps^m\}_{m=1}^K$ found by \cref{algo:ONGR}, and on the right some controls chosen according to the right-hand side of \eqref{eq:constructed_controls}.
\begin{figure}
	\includegraphics[width=0.45\textwidth]{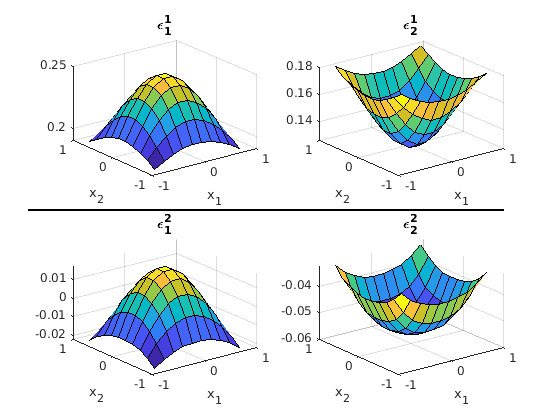}
	\hfill
	\includegraphics[width=0.45\textwidth]{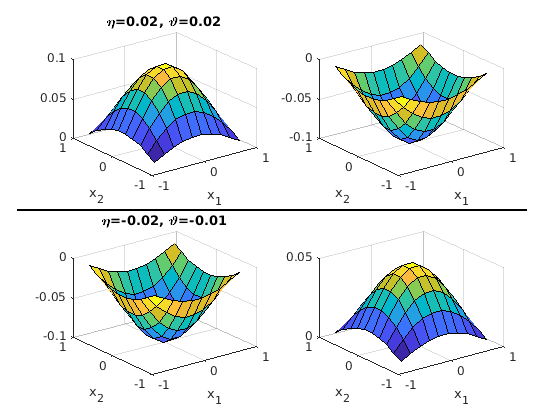}
	\caption{Controls functions. 
		Left: two pairs of controls obtained by \cref{algo:ONGR}; Right: two pairs of controls with the same structure as the right-hand side of \eqref{eq:constructed_controls} with different $\eta,\vartheta$. }
	\label{fig:results_controls}
\end{figure}
We observe that \cref{algo:ONGR} attempts to compute controls that mimic the structure of the right-hand side in \eqref{eq:constructed_controls} in order to reach different parts of the domain.

As a way of measuring the effectiveness of our procedure, we compare the behavior of the error when applying random controls for $P=5$.
For this, we consider the bilinear nonlinearity $G_\star^1(y_1,y_2) = 0.05y_1y_2$.
As controls, we (randomly) choose constant functions in $\Ead$ (cf. \eqref{eq:admissible_set_controls} for the definition of $\Ead$).
When we now try to reconstruct the nonlinearity, we obtain the results plotted in \cref{fig:identification_constant_controls,fig:identification_constant_controls_diag}.
In \cref{fig:identification_constant_controls}, we plot on the left the reconstructed nonlinearity and on the right the error $\mathfrak e$ together with the solution sets $S^m_\balpha$, $m \in [K]$, applying 19 random controls in $\Ead$ having a constant value.
We see clearly that the true nonlinearity is not recovered and that all solution sets lie on one straight line in the $y_1$-$y_2$-plane.
In \cref{fig:identification_constant_controls_diag}, we plot the desired and reconstructed nonlinearity on this straight line together with the error between them.
We see that there is a very good alignment.
Hence the final identification problem \eqref{eq:final_identification} is solved very precisely, however, this does not directly lead to a good reconstruction of the nonlinearity. 
The reason for this is, that the controls -- and with this the data -- were chosen poorly in the sense that they provide barely insights into the behavior of the nonlinearity. 
On the contrary, the controls that were found by our algorithm are able to reconstruct the nonlinearity very well as we have seen above.
Furthermore, some of them remind us of the structure of the controls that we constructed in \eqref{eq:constructed_controls} by theoretical considerations. 
This is visualized in \cref{fig:results_controls}.
Hence, the algorithm seems to provide an effective tool to generate robust controls automatically.
\begin{figure}
	\begin{subfigure}[l]{0.42\textwidth}
		\includegraphics[width=\textwidth]{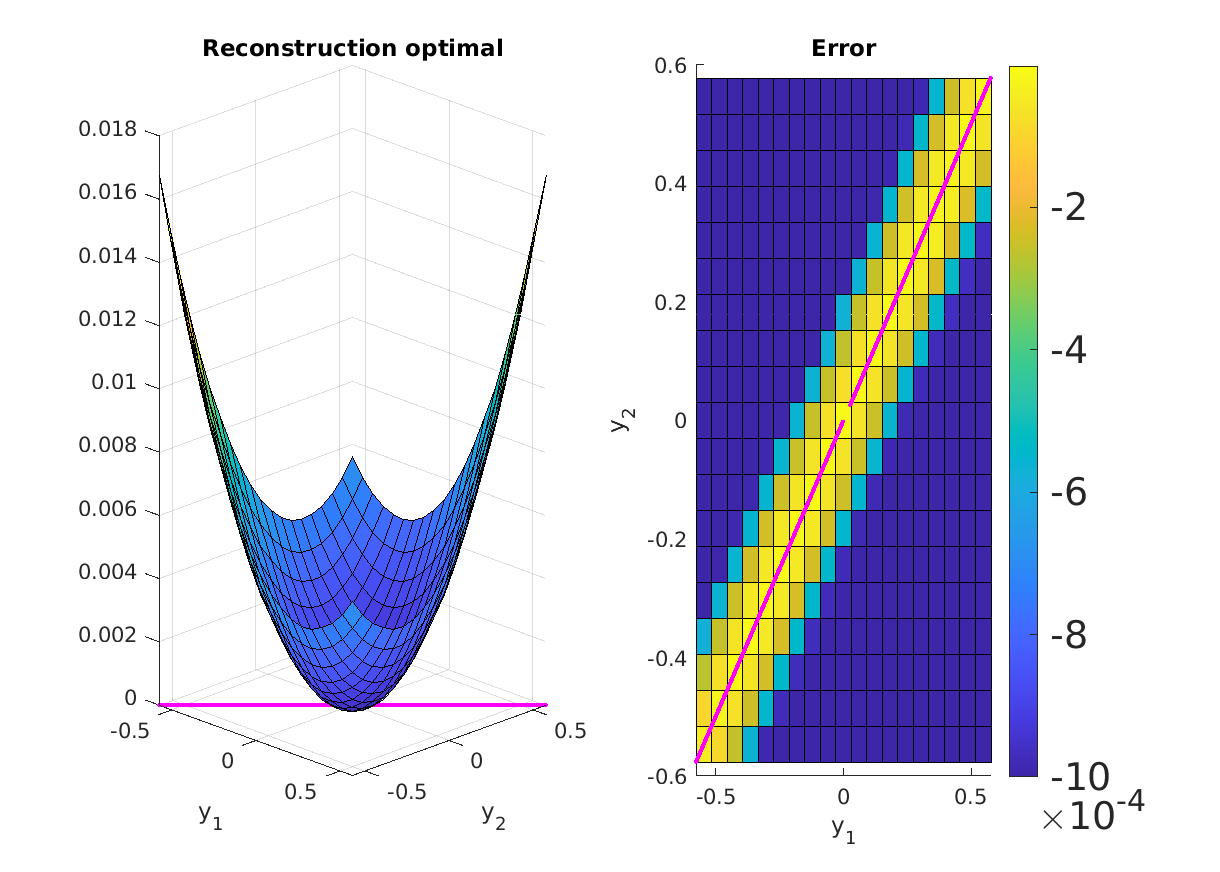}
		\caption{Plot of reconstruction and error.}
		\label{fig:identification_constant_controls}
	\end{subfigure}
	\hfill
	\begin{subfigure}[l]{0.37\textwidth}
		\includegraphics[width=\textwidth]{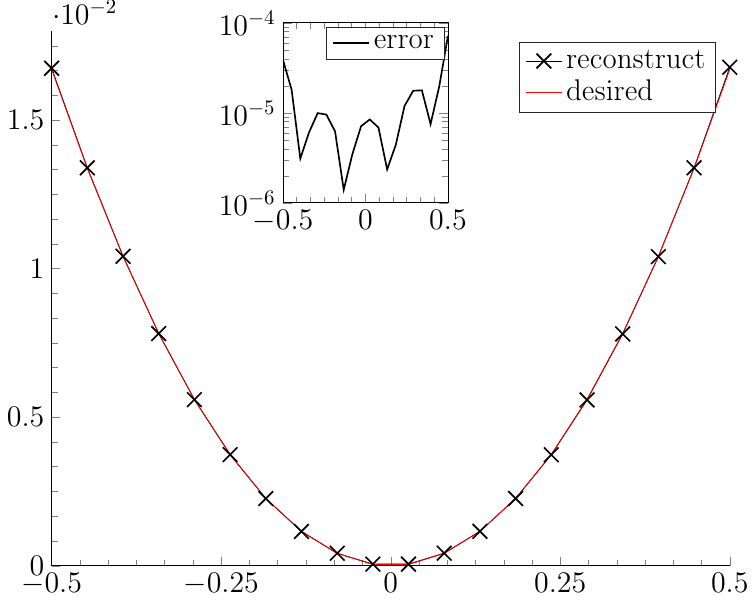}
		\caption{Cut through the diagonal of the domain.}
		\label{fig:identification_constant_controls_diag}
	\end{subfigure}
	\caption{Results of robustness test case (using random controls).}
\end{figure}
%

An important property of the algorithm is that it tries to make the problem more convex.
This is visualized in \cref{fig:convexify}.
In both figures, the functional is plotted for different values of the coefficient in front of the basis elements $y_1^2$ and $y_1y_2$.
Clearly, applying the optimal controls it is evident that the functional is convex and therefore the problem has a unique minimum.
For random controls, there seem to be infinitely local minima; at least the functional is quite flat in some directions.
See also Tables 6.1-6.3 in \cite{BuchwaldCiaramella2021GreedyReconstruction}.

\begin{figure}
	\begin{subfigure}[l]{0.4\textwidth}
		\includegraphics[width=\textwidth]{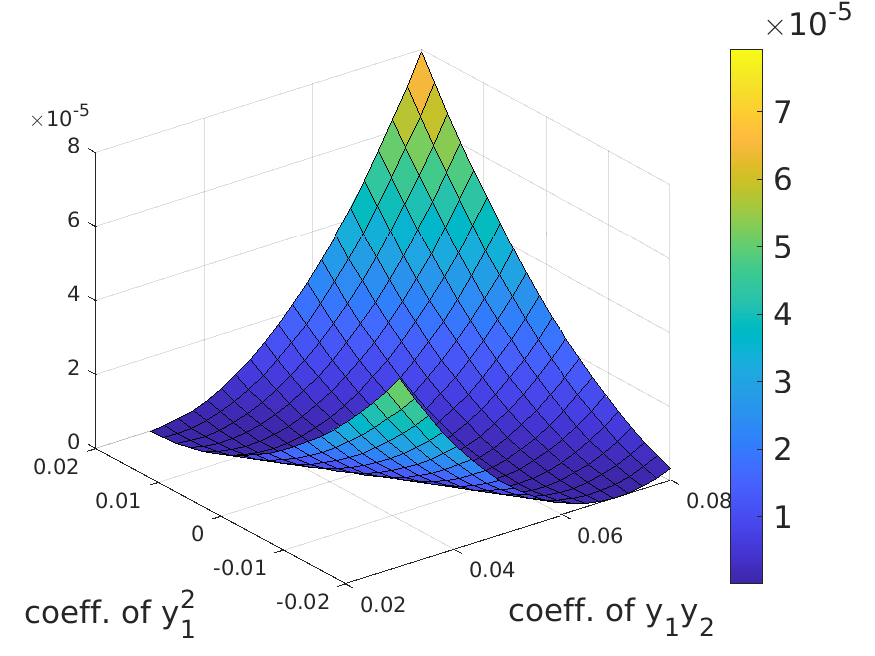}
	\end{subfigure}
	\hfill
	\begin{subfigure}[l]{0.40\textwidth}
		\includegraphics[width=\textwidth]{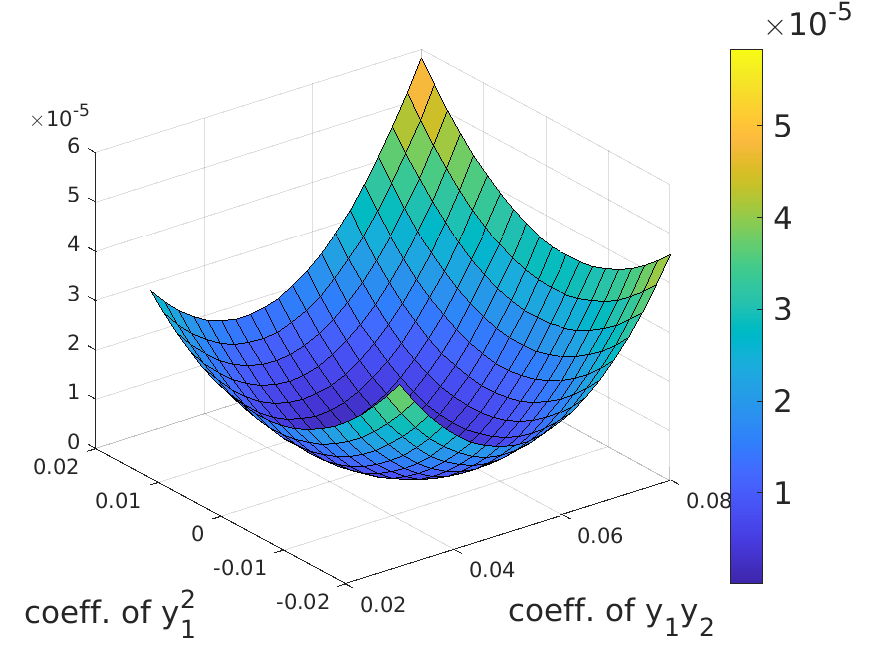}
	\end{subfigure}
	\caption{Making the problem more convex. Left: random controls; Right: optimal controls.}
	\label{fig:convexify}
\end{figure}

\subsection{Sinusoidal and exponential nonlinearity}
\label{sec:ExpNonlin}

Let us now consider the two remaining nonlinearities $G^2_\star(y) = 0.01\sin(2y_1)\sin(2y_2)$ and $G^3_\star(y)= 0.01\exp(2y_1)\exp(2y_2)$ as shown in \cref{subfig:sinusoidal,subfig:exponential}, respectively.
We repeat the experiments from \cref{sec:NumExperiments2} and obtain the results shown in \cref{fig:Sinusoidal_nonlinearity} and \cref{fig:Exponential_nonlinearity}, respectively.

%
\begin{figure}
	\begin{subfigure}[l]{\thirdsPlotsfactor\textwidth}
		\includegraphics[width=\textwidth]{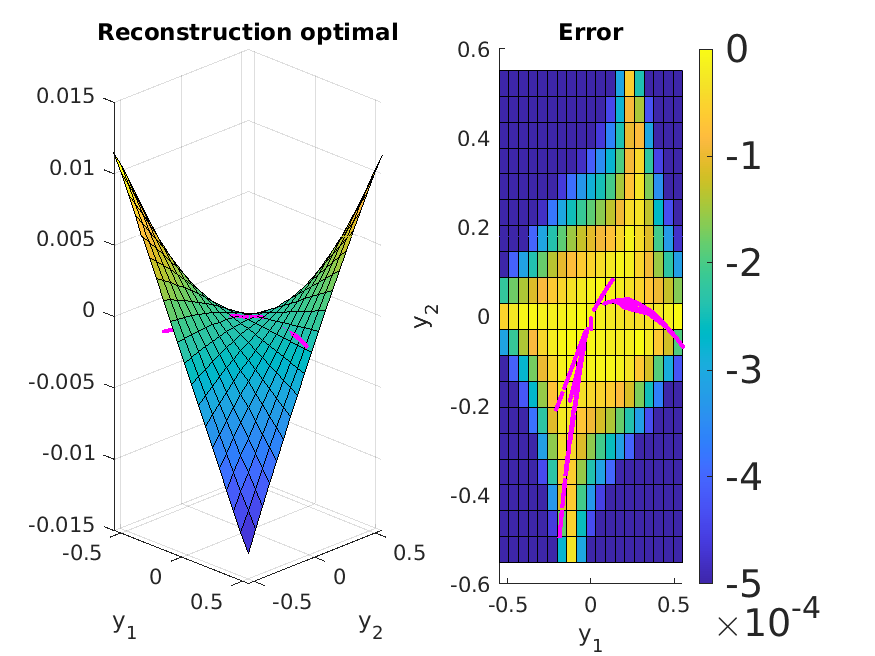}
	\end{subfigure}
	\hfill
	\begin{subfigure}[l]{\thirdsPlotsfactor\textwidth}
		\includegraphics[width=\textwidth]{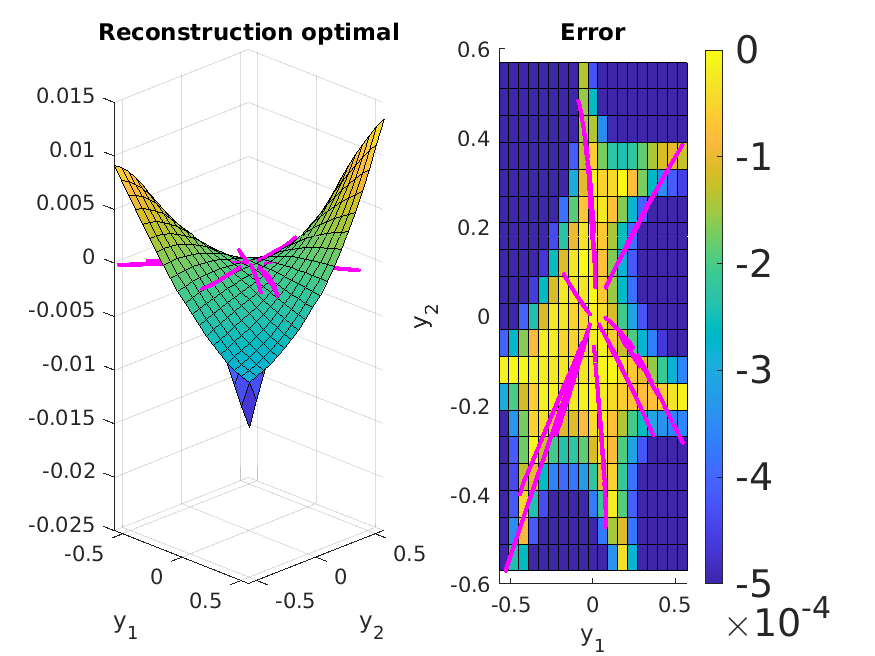}
	\end{subfigure}
	\hfill
	\begin{subfigure}[l]{\thirdsPlotsfactor\textwidth}
		\includegraphics[width=\textwidth]{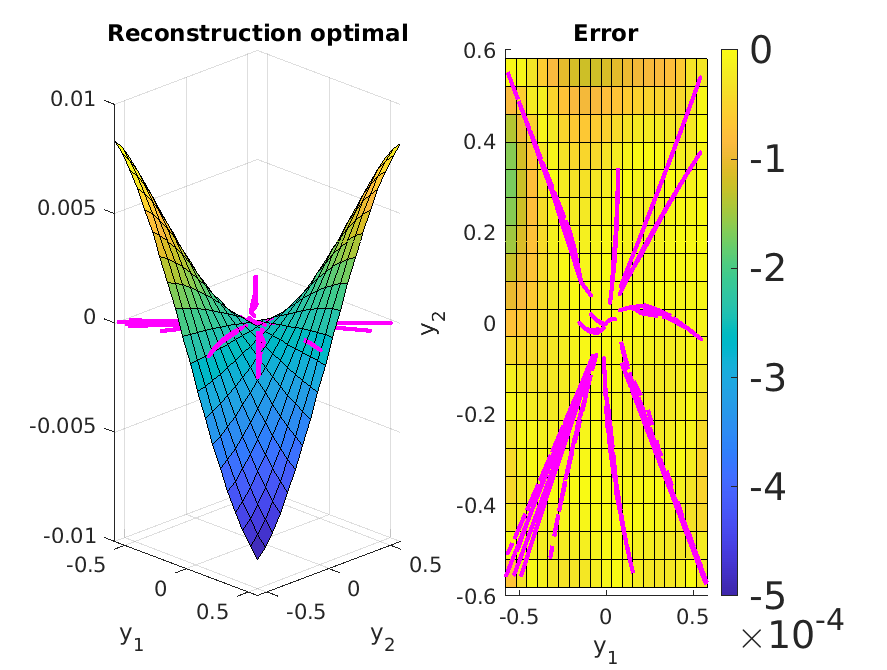}
	\end{subfigure}
	\caption{Reconstruction and error plots (cf. \eqref{eq:definition_error}) for $P=2,3,5$ and the sinusoidal nonlinearity with the solutions curves in magenta.}
	\label{fig:Sinusoidal_nonlinearity}
\end{figure}

%
\begin{figure}
	\begin{subfigure}[l]{\thirdsPlotsfactor\textwidth}
		\includegraphics[width=\textwidth]{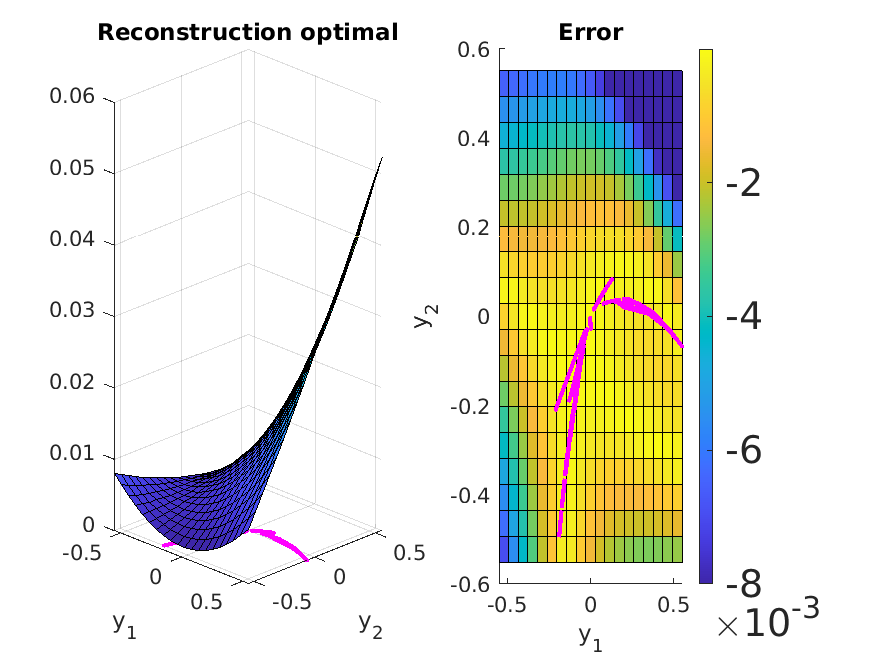}
	\end{subfigure}
	\hfill
	\begin{subfigure}[l]{\thirdsPlotsfactor\textwidth}
		\includegraphics[width=\textwidth]{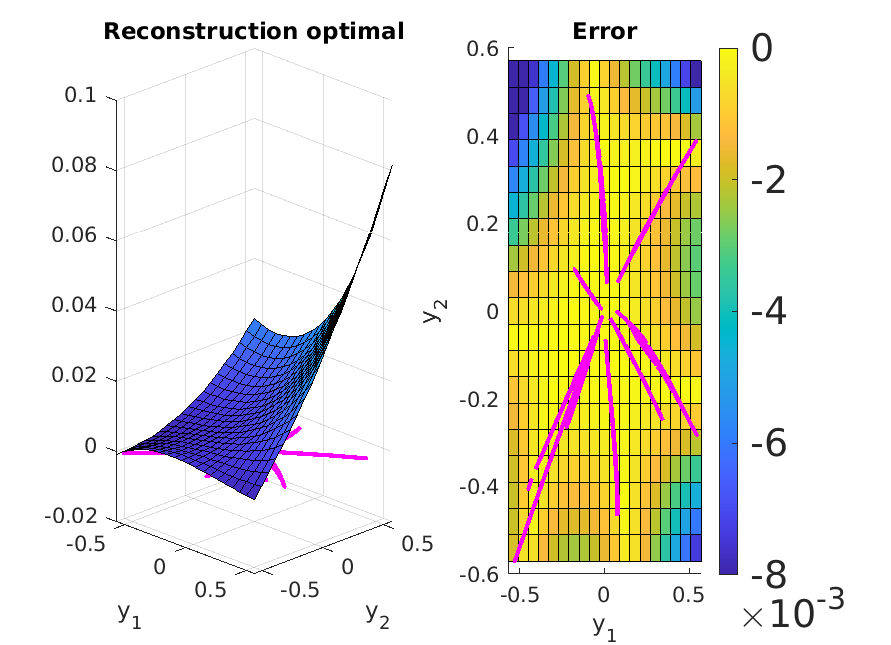}
	\end{subfigure}
	\hfill
	\begin{subfigure}[l]{\thirdsPlotsfactor\textwidth}
		\includegraphics[width=\textwidth]{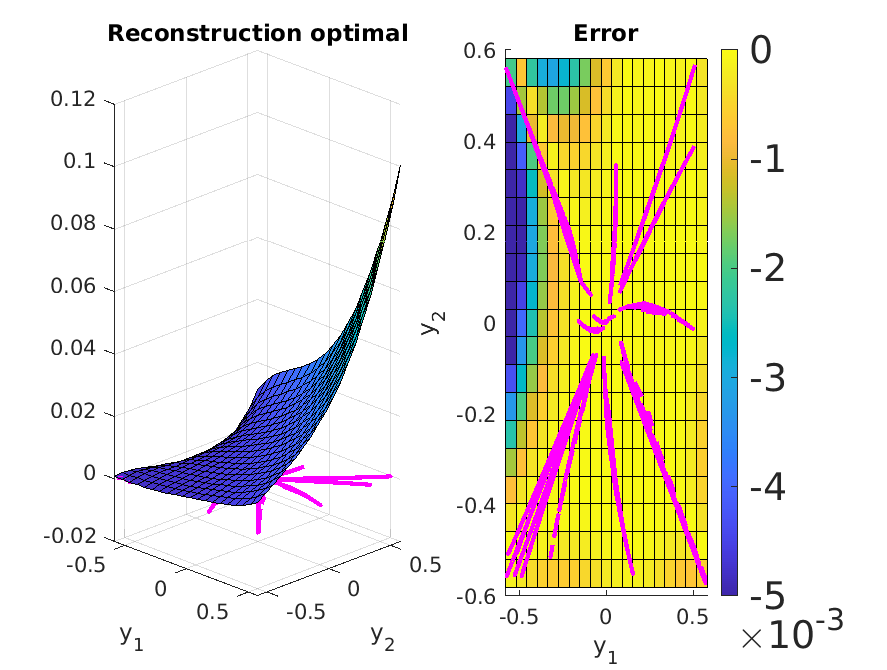}
	\end{subfigure}
	\caption{Reconstruction and error plots (cf. \eqref{eq:definition_error}) for $P=2,3,5$ and the exponential nonlinearity with the solutions curves in magenta.}
	\label{fig:Exponential_nonlinearity} 
\end{figure}

From the error plots in \cref{fig:Exponential_nonlinearity}, we observe that we need $P=5$ to have a sensible reconstruction of the exponential nonlinearity. 
This is to be expected since the Taylor expansion of the exponential function consists of the sum of all monomials.
Furthermore, we see more clearly that the error on each coefficient in the Taylor expansion goes down when increasing the order of reconstruction $P$ (cf. \cref{fig:Error_Taylor_coeffs}).
The reason is that in every iteration functions that are part of the Taylor function enter the reconstruction.

Since within $\Omega_\Sc$ the sinusoidal nonlinearity is very similar to a bilinear one, the error is in principle smaller than for the exponential nonlinearity. 
We also observe that the error in \cref{fig:Sinusoidal_nonlinearity} is worse for $P=3$ than for $P=2$.
This can be explained by the fact that uneven functions enters for $P=3$ compared to $P=2$.
Notice also that the error in the first subfigure in \cref{fig:Sinusoidal_nonlinearity} is symmetric with respect to a diagonal line.
Hence the error that this function imposes outside of the solution sets might be quite high.

\subsection{Error in Taylor coefficients}
\label{sec:ErrorTaylorCoeff}

To round up our numerical experiments, we discuss a different way of measuring the error of the final reconstruction: the absolute value of the difference of the coefficient of each monomial in the Taylor expansion of the true nonlinearity and the reconstructed nonlinearity, respectively.
	
More precisely, let the Taylor expansion up to order $d^2$ around the origin given by
\begin{align}
    \mathcal{T}_{d^2}(y_1,y_2) = 
    \sum_{i_1=0}^d \sum_{i_2=0}^d t_{i_1,i_2} y_1^{i_1}\,y_2^{i_2},
    \qquad\qquad
    t_{i_1,i_2} = \frac{1}{i_1! \, i_2!} \left(\frac{\partial^{i_1+i_2}}{\partial_{y_1}^{i_1}\, \partial_{y_2}^{i_2}} g(0,0) \right).
    \label{eq:diff_taylor_coeff}
\end{align} 
Now, we compare the coefficients $\balpha = (\alpha_{i_1,i_2})_{i_1,i_2=0}^d$ from our reconstruction of the polynomials given in \eqref{eq:Polynomials_definition} with the coefficients $t_{i_1,i_2}$ of the Taylor expansion given in \eqref{eq:diff_taylor_coeff}. 
Since the Taylor expansion gives rise to the same basis that we consider in our algorithm, it is possible to compare the corresponding coefficients.
The difference in each coefficient is shown in the following plots in increasing order of the coefficients.
In \cref{fig:Error_Taylor_coeffs}, the error in the Taylor coefficients (up to the corresponding approximation order) is plotted. 
In all of these plots, we observe a saturation effect of the error for higher orders of degree.
This can be seen most strongly for $P=3$ and $P=5$.
The reason for this is that for a higher order of polynomial degree the influence of the corresponding monomial growths with the distance to the origin. Since we consider here a quite small domain, it is harder to reconstruct the correct coefficient.

\begin{figure}
	\begin{subfigure}[l]{\thirdsPlotsfactor\textwidth}
		\includegraphics[width=\textwidth]{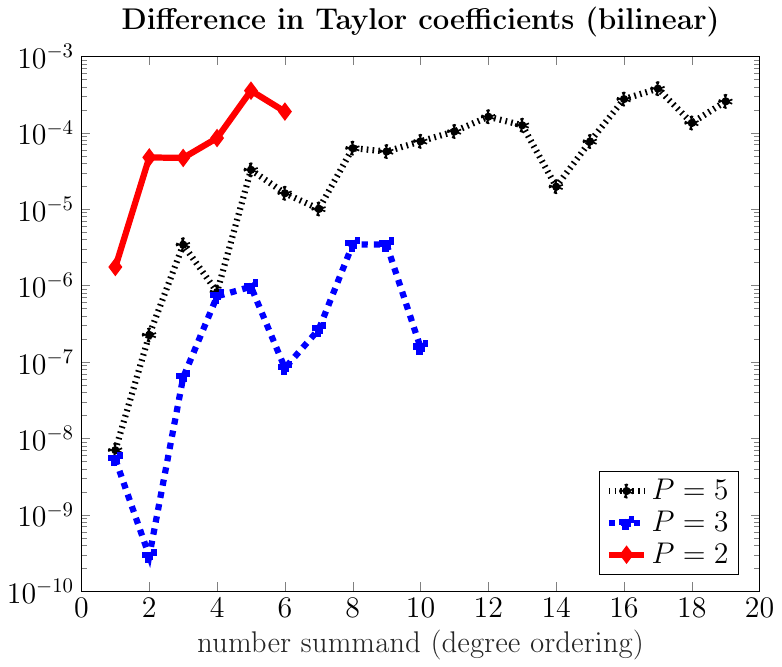}
	\end{subfigure}
    \hfill
	\begin{subfigure}[l]{\thirdsPlotsfactor\textwidth}
		\includegraphics[width=\textwidth]{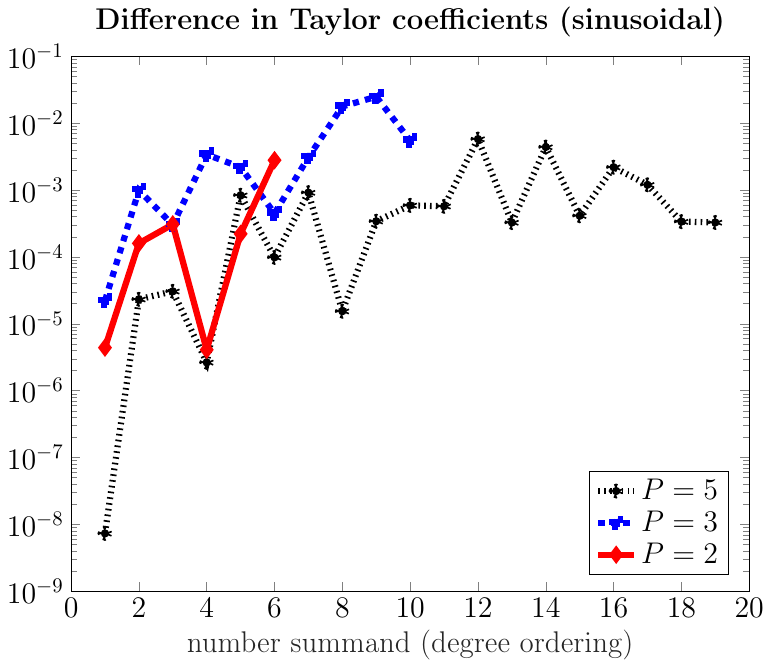}
	\end{subfigure}
	\hfill
	\begin{subfigure}[l]{\thirdsPlotsfactor\textwidth}
		\includegraphics[width=\textwidth]{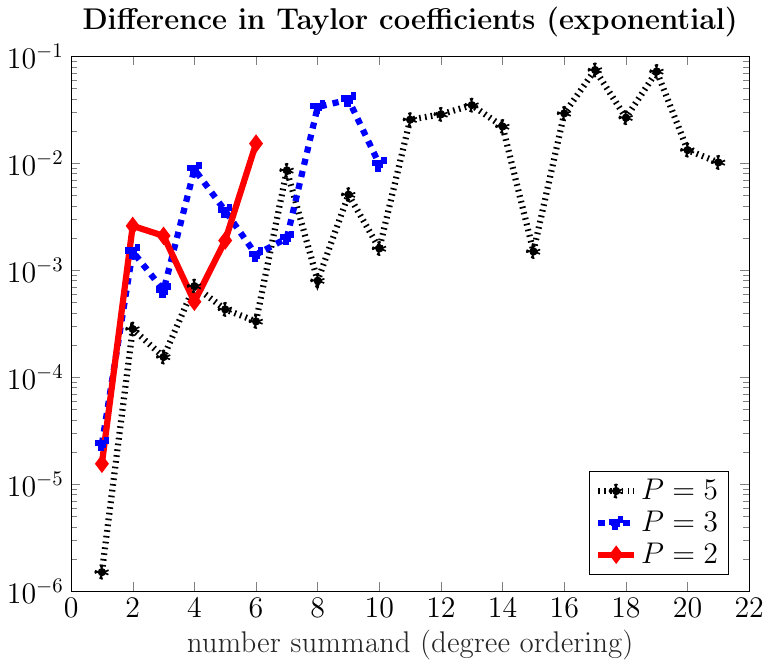}
	\end{subfigure}
	
	\caption{Difference in coefficients of the (true) Taylor expansion (cf. \eqref{eq:diff_taylor_coeff}) and the reconstructed nonlinearity for $P=2,3,5$ for bilinear, sinusoidal and exponential nonlinearity.}
	\label{fig:Error_Taylor_coeffs}
\end{figure}

\section{Conclusion}
\label{sec:Conclusion}

In this work, we presented a greedy reconstruction algorithm to identify unknown operators in nonlinear elliptic models. 
We proved the Lipschitz continuity of the parameter-to-state map and its inverse.
We performed several numerical experiments that successfully validate our proposed scheme.

This paper represents a significant step towards untangling the mysteries of unknown nonlinear operators within semilinear elliptic models. 
By harnessing the power of optimal control and active learning, we pave the way for a deeper understanding and more accurate predictions of complex physical phenomena governed by these models.

\section*{Acknowledgments}

We would like to express our gratitude to Jacob Körner (University of Würzburg), Jörg Weber (University of Vienna), and Behzad Azmi (University of Konstanz) for their beneficial remarks and fruitful discussions.
We would also like to express our gratitude to the anonymous referees for their helpful questions and remarks.

\textbf{Competing interests.}
The authors have no relevant financial or non-financial interests to disclose.

\textbf{Author contributions.}
All four authors  contributed equally to the present manuscript.

\bibliographystyle{acm}

\end{document}